\documentclass[12pt,reqno]{amsart}

\usepackage{amsfonts,amsmath,amsthm}
\usepackage{amssymb,epsfig}
\usepackage{cases}
\usepackage[usenames]{color}
\usepackage{enumerate} 
\usepackage{caption}

\usepackage[text={425pt,650pt},centering]{geometry}
\usepackage{graphicx}
\usepackage{epsfig}
\usepackage{caption}
\usepackage{tikz}
\usepackage{color} 
\definecolor{vert}{rgb}{0,0.6,0}

\theoremstyle{plain}
\newtheorem{thm}{Theorem}[section]

\newtheorem{defn}{Definition}
\newtheorem{quest}{Question}

\newtheorem{lem}[thm]{Lemma}
\newtheorem{cor}[thm]{Corollary}

\newtheorem{conj}{Conjecture}

\theoremstyle{remark}
\newtheorem{rem}{\bf{Remark}}
\numberwithin{equation}{section}

\newcommand{\N}{\mathbb{N}}
\newcommand{\bP}{\mathbb{P}}
\newcommand{\R}{\mathbb{R}}

\newcommand{\T}{\mathbb{T}}
\newcommand{\Z}{\mathbb{Z}}

\newcommand{\cF}{\mathcal{F}}

\newcommand{\BUC}{{\rm BUC\,}}

\newcommand{\del}{\delta}
\newcommand{\ep}{\varepsilon}

\newcommand{\lam}{\lambda}
\newcommand{\sig}{\sigma}
\newcommand{\om}{\omega}

\newcommand{\Om}{\Omega}

\newcommand{\ol}{\overline}

\begin{document}

\title[Min-max formulas  and other properties of $\ol{H}$]{Min-max formulas  and other properties of certain classes of nonconvex effective Hamiltonians}

\author{Jianliang Qian}
\address[Jianliang Qian]
{
Department of Mathematics and Department of Computational Mathematics, Science and Engineering, 
Michigan State University, East Lansing, MI 48824 , USA}
\email{qian@math.msu.edu}

\author{Hung V. Tran}
\address[Hung V. Tran]
{
Department of Mathematics, 
University of Wisconsin Madison, Van Vleck hall, 480 Lincoln drive, Madison, WI 53706, USA}
\email{hung@math.wisc.edu}

\author{Yifeng Yu}
\address[Yifeng Yu]
{
Department of Mathematics, 
University of California, Irvine, 410G Rowland Hall, Irvine, CA 92697, USA}
\email{yyu1@math.uci.edu}

\thanks{
The  work of JQ is partially supported by NSF grants 1522249 and 1614566,
the work of HT is partially supported by NSF grant DMS-1615944,
the work of YY is partially supported by NSF CAREER award \#1151919.
}

\date{}


\maketitle 

\begin{abstract}
This paper is the first attempt to systematically study properties of the effective Hamiltonian $\ol{H}$  arising in the  periodic homogenization of some  coercive but nonconvex Hamilton-Jacobi equations.  Firstly,  we introduce a new and robust decomposition method to obtain min-max formulas for a class of nonconvex $\ol{H}$.  
Secondly,  we   analytically and numerically investigate other related interesting  phenomena, such as ``quasi-convexification" and breakdown of symmetry, of $\ol{H}$ from other typical nonconvex Hamiltonians.  
Finally,  in the appendix, we show that our new method and those a priori formulas from the periodic setting can be used to obtain stochastic homogenization  for same class of nonconvex Hamilton-Jacobi equations.  Some conjectures and problems are also proposed. 
\end{abstract}

\section{Introduction}

\subsection{Overview}
Let us describe the periodic homogenization theory of Hamilton-Jacobi equations.
For each $\ep>0$, let $u^\ep \in C(\R^n \times [0,\infty))$ be the viscosity solution to
\begin{equation}\label{HJ-ep}
\begin{cases}
u^\ep_t + H(Du^\ep) - V\left(\frac{x}{\ep}\right)=0 \quad &\text{ in } \R^n \times (0,\infty),\\
u^\ep(x,0)=g(x) \quad &\text{ on } \R^n.
\end{cases}
\end{equation}
Here, the Hamiltonian $H(p) -V(x)$ is of separable form with $H \in C(\R^n)$, which is coercive (i.e., $\lim_{|p| \to \infty} H(p)=+\infty$),
and $V \in C(\R^n)$, which is $\Z^n$-periodic.
The initial data $g\in \BUC(\R^n)$, the set of bounded, uniformly continuous functions on $\R^n$.

It was proven by Lions, Papanicolaou and Varadhan \cite{LPV} that $u^\ep$ converges to $u$ locally uniformly on $\R^n \times [0,\infty)$ as $\ep \to 0$,
and $u$ solves the effective equation
\begin{equation}\label{HJ-hom}
\begin{cases}
u_t +\ol{H}(Du)=0 \quad &\text{ in } \R^n \times (0,\infty),\\
u(x,0)=g(x) \quad &\text{ on } \R^n.
\end{cases}
\end{equation}
The effective Hamiltonian $\ol{H} \in C(\R^n)$ is determined in a nonlinear way by $H$ and $V$ through the cell problems as following.  
For each $p \in \R^n$, it was shown in \cite{LPV} that there exists a unique constant $\ol H(p)\in \R$ such that the following cell problem has a continuous viscosity solution
\begin{equation} \label{E-p}
H(p+Dv) - V(x) = \ol H(p) \quad \text{ in } \T^n,
\end{equation}
where $\T^n$ is the $n$-dimensional flat torus $\R^n/\Z^n$.

Although there is a vast literature on homogenization of Hamilton-Jacobi equations in different settings after \cite{LPV},  characterizing the shape of $\ol H$ remains largely open even in  basic situations. 
Let us summarize quickly what is known in the literature about $\ol{H}$.
It is not hard to see that $\ol{H}$ is coercive thanks to the coercivity of $H$.
If one assumes furthermore that $H$ is convex, then $\ol{H}$ is also convex and the graph of $\ol{H}$ can contain some flat parts (i.e., $\left\{\ol H=\min \ol H\right\}$ has interior points). See \cite{LPV} and the works of Concordel \cite{Con1,Con2}. Furthermore, in this convex setting, we have the following representation formula, thanks to the results of 
Contreras, Iturriaga, Paternain and Paternain \cite{CIPP}, and Gomes \cite{Gom},
\begin{equation}\label{convex-min-max}
\ol{H}(p) = \inf_{\phi \in C^1(\T^n)} \max_{x\in \T^n} \left ( H(p+D\phi(x))- V(x) \right).
\end{equation}
Note that  the above representative formula still holds if $H$ is quasiconvex (level-set convex),
in which case $\ol{H}$ is also quasiconvex.  
More interestingly,  in case $n=2$, $H(p)=|p|^2$ and $V\in C^{\infty}(\Bbb T^2)$,  
a deep result of Bangert \cite{B2} says that the level curve $\left\{\ol H=c\right\}$ for every $c>-\min V$ must contain line segments (i.e., not strictly convex) unless $V$ is a constant function.   Bangert's result relies on detailed information about the structure of Aubry-Mather sets in two dimension (\cite{B1}).  See also Jing, Tran, Yu \cite{JTY} for discussion regarding locations of  line segments of the level curves for Ma\~n\'e type Hamiltonians.  

The first numerical computation of effective Hamiltonians  is due to  Qian \cite{Qi} based on the so called {\it big-T method}, that is,  $\ol H(p)=-\lim_{t\to \infty} {w(x,t)\over t}$,
where $w(x,t)$ is the unique viscosity solution to 
\[
\begin{cases}
w_t+H(Dw)-V(x)=0 \quad &\text{ in $\R^n\times (0, \infty)$,}\\
w(x,0)=p\cdot x  \quad &\text{ on $\R^n$}.
\end{cases}
\]
For other numerical schemes,  we refer to Gomes, Oberman \cite{GO}, 
Falcone, Rorro \cite{FR}, 
Achdou, Camilli, Capuzzo-Dolcetta \cite{ACC},
Oberman, Takei, Vladimirsky \cite{OTV}, Luo, Yu, Zhao \cite{LYZ}
and the references therein.

It is worth mentioning that cell problem \eqref{E-p} and representation formula \eqref{convex-min-max} 
appear also in weak KAM theory (see E \cite{E}, Evans, Gomes \cite{EvG},
Fathi \cite{Fa} for the convex case, and Cagnetti, Gomes, Tran \cite{CGT} for the nonconvex case).  
In fact,  a central goal of the weak KAM theory is to find information of underlying dynamical system encoded in the effective Hamiltonian.

In the case where $H$ is nonconvex, to identify   the shape of $\ol{H}$ is highly nontrivial even in  the one dimensional space. 
This was settled only very recently by  Armstrong, Tran, Yu \cite{ATY2}, and Gao \cite{Gao}.  
One fundamental feature obtained is the  ``{\it quasi-convexification}" phenomenon, that is, 
the effective Hamiltonian $\ol H$ becomes quasiconvex (level-set convex)   when the oscillation of $V$ is large enough. 
See Section \ref{sec:quasi} for more precise statements. 
In multi-dimensional spaces,
 Armstrong, Tran, Yu \cite{ATY1} obtained a qualitative shape of $\ol{H}$ for a representative case where $H(p)=(|p|^2-1)^2$.
Other than \cite{ATY1}, very little is known about finer properties of nonconvex $\ol{H}$ in multi-dimensional spaces, partly due to the extreme complexity of dynamics associated with nonconvex Hamiltonians.  Furthermore, as far as the authors know, there is no numerical study of $\ol{H}$ in this case.

Let us also call attention to an extrinsic way to study $\ol{H}$ via inverse problems. 
See Luo, Tran, Yu \cite{LTY}.

\subsection{Main results}
Reducing a complex quantity to relatively simpler objects is a very natural and common idea in  mathematics.  
For a class of nonconvex Hamiltonians $H$,  we introduce a new decomposition method to obtain min-max type representation formulas for $\ol{H}$.  
These formulas  consist of   effective Hamiltonians of quasiconvex Hamiltonians which are presumably less challenging  to analyze.  
The most general statement is given by inductive formulas (Theorem \ref{Maintheorem}).  
Two specific (but important) cases of Theorem \ref{Maintheorem}  are provided in Theorem  \ref{thm:rep1} and Lemma \ref{thm:rep2}.   
One immediate corollary is the evenness of $\ol H$ associated with a certain class of radially symmetric Hamiltonians, which is otherwise not obvious at all.  Given the vast  variety of  nonconvex functions,    our surgical approach is only a preliminary step toward understanding the shape of nonconvex $\ol H$.  In Section 2.4, we present some natural obstacles to decomposing a nonconvex $\ol H$. In particular, there is a connection between ``non-decomposability" and loss of evenness.

As another interesting application,  the method and the representation formulas are robust enough that 
we are also able to prove stochastic homogenization for the same class of nonconvex $H$ in the appendix. 
For instance,   Theorem \ref{thm:random}   includes the result in \cite{ATY1} as a special case with a much shorter proof.  
The detailed discussion on this (including a brief overview of stochastic homogenization) is left to the appendix.
We would like to point out  that a priori identification of shape of $\ol H$ is currently  the only  available way to  tackle  homogenization of  nonconvex Hamilton-Jacobi equations  in general stationary ergodic setting. 

In Section 3,  we  provide various numerical computations of $\ol{H}$ in multi-dimensional spaces for general radially symmetric Hamiltonians and a double-well type Hamiltonian.  
These provide insights on how the changes of potential energy $V$ affect the changes in shape of effective Hamiltonian $\ol{H}$.  The important  ``quasi-convexification" phenomenon is observed  in multi-dimensional cases as well.  Nevertheless,  verifying  it rigorously seems to be quite challenging.  Interesting connections between decomposition,  loss of evenness and quasi-convexification are demonstrated  in  Section 2.4 and  Remark \ref{rem:conj}.  Several open problems are provided based on the numerical  evidences we have in this section.

\section{Min-max formulas} \label{sec:min-max}
\subsection{Basic case}
The setting is this. Let $H=H(p): \R^n \to \R$ be a continuous, coercive Hamiltonian such that
\begin{itemize}
\item[(H1)] $\min_{\R^n} H=0$ and there exists a bounded domain $U \subset \R^n$ such that 
\[
\{H=0\}=\partial U.
\]
\item[(H2)] $H(p)=H(-p)$ for all $p\in \R^n$.
\item[(H3)] There exist $H_1, H_2: \R^n \to \R$ such that $H_1, H_2$ are continuous and 
\[
H=\max\{H_1,H_2\}.
\]
Here, $H_1$ is coercive, quasiconvex, even ($H_1(p)=H_1(-p)$ for all $p\in \R^n$), $H_1=H$ in $\R^n \setminus U$
and $H_1 <0$ in $U$.
The function $H_2$ is quasiconcave, $H_2=H$ in $U$, $H_2<0$ in $\R^n \setminus U$ and $\lim_{|p| \to \infty} H_2(p)=-\infty$.
\end{itemize}
It is easy to see that  any  $H$ satisfying (H1)--(H3)  can be written as $H(p)=|F(p)|$ for some  even, coercive quasiconvex function $F$ such that $\min_{\R^n} F<0$. Below is the first decomposition result. 

\begin{thm}\label{thm:rep1}
Let $H \in C(\R^n)$ be a Hamiltonian satisfying {\rm (H1)--(H3)}.
Let $V \in C(\T^n)$ be a potential energy with $\min_{\T^n} V=0$.

Assume that $\ol{H}$ is the effective Hamiltonian corresponding to $H(p)-V(x)$.
Assume also that $\ol{H}_i$ is the effective Hamiltonian corresponding to $H_i(p)- V(x)$ for $i=1,2$.
Then
\[
\ol{H} = \max\{\ol{H}_1, \ol{H}_2,0\}.
\]
In particular,  $\ol H$ is even. 
\end{thm}

We would like to point out that the evenness of $ \ol{H}$ will be used later and is not  obvious at all  although $H$ is  even.  See the discussion in Subsection \ref{subsec:even} for this subtle issue.

\begin{proof}
We proceed in few steps.
\smallskip

\noindent {\bf Step 1.} It is straightforward that $0 \leq \ol{H}(p) \leq H(p)$ for all $p \in \R^n$. In particular, 
\begin{equation}\label{basic-1}
\ol{H}(p)=0 \quad \text{ for all } p \in \partial U.
\end{equation}
Besides, as $H_i \leq H$, we get $\ol{H}_i \leq \ol{H}$. Therefore, 
\begin{equation}\label{basic-2}
\ol{H} \geq \max\left\{\ol{H}_1, \ol{H}_2,0\right\}.
\end{equation}
It remains to prove the reverse inequality of \eqref{basic-2} in order to get the conclusion.
\smallskip

\noindent {\bf Step 2.} Fix $p\in \R^n$. 
Assume now that $\ol{H}_1(p) \geq \max\{\ol{H}_2(p),   0\}$.
We will show that $\ol{H}_1(p) \geq \ol{H}(p)$.

 Since $H_1$ is quasiconvex and even, we use the inf-max representation formula for $\ol{H}_1$
 (see \cite{AS3,DaSi, Na}) to get that
\begin{align*}
\ol{H}_1(p) &= \inf_{\phi \in C^1(\T^n)} \max_{x \in \T^n} \left ( H_1(p+D\phi(x)) - V(x) \right)\\
&=\inf_{\phi \in C^1(\T^n)} \max_{x \in \T^n} \left ( H_1(-p-D\phi(x)) - V(x) \right)\\
&=\inf_{\psi \in C^1(\T^n)} \max_{x \in \T^n} \left ( H_1(-p+D\psi(x)) - V(x) \right)
=\ol{H}_1(-p).
\end{align*}
Thus, $\ol{H}_1$ is even.
Let $v(x,-p)$ be a solution to the cell problem
\begin{equation}\label{basic-3}
H_1(-p+Dv(x,-p)) - V(x) = \ol{H}_1(-p)=\ol{H}_1(p) \quad \text{ in } \T^n.
\end{equation}
Let $w(x)=-v(x,-p)$. For any $x\in \T^n$ and $q \in D^+ w(x)$, we have $-q \in D^- v(x,-p)$ and hence,
in light of \eqref{basic-3} and the quasiconvexity of $H_1$ (see \cite{BJ}),
\[
\ol{H}_1(p)=H_1(-p-q) - V(x) = H_1(p+q) - V(x).
\]
We thus get $H_1(p+q) = \ol{H}_1(p) + V(x) \geq 0$, and therefore, $H(p+q) = H_1(p+q)$.
This yields that $w$ is a viscosity subsolution to
\[
H(p+Dw) - V(x) = \ol{H}_1(p) \quad \text{ in } \T^n.
\]
Hence, $\ol{H}(p) \leq \ol{H}_1(p)$.
\smallskip

\noindent {\bf Step 3.}
Assume now that $\ol{H}_2(p) \geq \max\{\ol{H}_1(p), \ 0\}$. 
By using similar arguments as those in the previous step (except that we use $v(x,p)$ instead of $v(x,-p)$ due to the quasiconcavity of $H_2$), 
we can show that $\ol{H}_2(p) \geq \ol{H}(p)$.

\smallskip

\noindent {\bf Step 4.}
Assume that $\max\left\{\ol{H}_1(p), \ol{H}_2(p)\right\}<0$. We now show that $\ol{H}(p)=0$ in this case.
Thanks to \eqref{basic-1} in Step 1, we may assume that $p \notin \partial U$.

For $\sig \in [0,1]$ and $i=1,2$, let $\ol{H}^\sig$, $\ol{H}^\sig_i$ be the effective Hamiltonians corresponding to $H(p) - \sig V(x)$, $H_i(p)-\sig V(x)$, respectively. 
It is clear that
\begin{equation}\label{s4-1}
0 \leq \ol{H}^1=\ol{H} \leq \ol{H}^\sig \quad \text{ for all } \sig \in [0,1].
\end{equation}
By repeating Steps 2 and 3 above, we get 
\begin{equation} \label{s4-2}
\text{For $p\in\R^n$ and $\sig \in [0,1]$, if $\max\left\{\ol{H}^\sig_1(p),\ol{H}^\sig_2(p)\right\}=0$, then $\ol{H}^\sig(p)=0$.}
\end{equation}
We only consider the case $p \notin \ol{U}$ here.
The case $p \in U$ is analogous. Notice that
\[
H(p)= H_1(p) = \ol{H}^0_1(p)>0 \quad \text{and} \quad \ol{H}_1(p)= \ol{H}^1_1(p) <0.
\]
By the continuity of $\sig \mapsto \ol{H}^\sig_1(p)$, there exists $s \in (0,1)$ such that $\ol{H}^s_1(p)=0$.
Note furthermore that, as $p \notin \ol{U}$, $\ol{H}_2^s(p) \leq H_2(p) <0$.
These, together with \eqref{s4-1} and \eqref{s4-2}, yield the desired result.
\end{proof}

\begin{rem}\label{rem:step4}
We emphasize that Step 4 in the above proof is important.
It plays the role of a ``patching" step, which helps glue $\ol{H}_1$
and $\ol{H}_2$ together.

It is worth noting that the representation formula in Theorem \ref{thm:rep1} still holds in case $H$ is not even in $U$.
In fact, we do not use this point at all in the proof. We only need it to deduce that $\ol{H}$ is even.
\end{rem}

Assumptions (H1)--(H3) are general and a bit complicated.
A simple situation where (H1)--(H3) hold is a radially symmetric case where $H(p)=\psi(|p|)$,
and  $\psi \in C([0,\infty), \R)$ satisfying
\begin{equation}\label{psi-con}
\begin{cases}
\psi(0)>0, \ \psi(1)=0, \  \lim_{r \to \infty} \psi(r)=+\infty,\\
\psi \text{ is strictly decreasing in $(0,1)$ and is strictly increasing in $(1,\infty)$}.
\end{cases}
\end{equation}
Let $\psi_1, \psi_2 \in C([0,\infty), \R)$ be such that
\begin{equation}\label{psi-1-2}
\begin{cases}
\psi_1=\psi  \text{ on } [1,\infty), \text{ and } \psi_1 \text{ is strictly increasing on } [0,1],\\
\psi_2=\psi \text{ on } [0,1], \ \psi_2  \text{ is strictly decreasing on } [1,\infty), \text{ and } \lim_{r \to \infty} \psi_2(r)=-\infty.
\end{cases}
\end{equation}
See Figure \ref{fig1} below.   Set $H_i(p)= \psi_i(|p|)$ for $p \in \R^n$, and for $i=1,2$.
It is clear that (H1)--(H3) hold provided that \eqref{psi-con}--\eqref{psi-1-2} hold.
\begin{figure}[h]
\begin{center}
\begin{tikzpicture}\label{fig1}

\draw[->] (-0.5,0)--(4,0);
\draw[->] (0,-2)--(0,4);
\draw (4,-0.2) node {$r$};
\draw (1,-0.2) node {$1$};

\draw[dashed, blue, ultra thick] (0,-0.5)--(1,0);
\draw[blue] (0.2,-0.7) node{$\psi_1$};

\draw[dashed, red, ultra thick] plot [smooth] coordinates {(1,0) (2,-0.7) (3,-1.8)};
\draw[red] (3.2, -1.9) node{$\psi_2$};

\draw[thick] plot [smooth] coordinates {(0,1) (1,0) (2,1) (3,4)};
\draw (3.2,4) node {$\psi$};
\end{tikzpicture}
\caption{Graphs of $\psi, \psi_1, \psi_2$}
\label{fig1}
\end{center}
\end{figure}
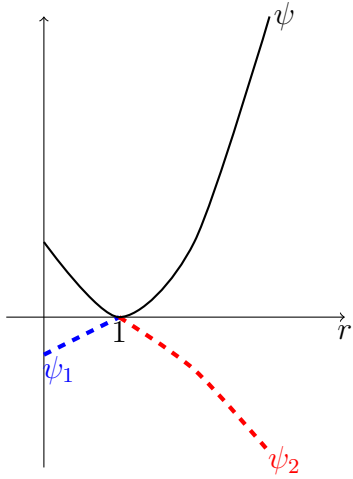
An immediate consequence of Theorem \ref{thm:rep1} is
\begin{cor}\label{cor:rep1}
Let $H(p)=\psi(|p|)$, $H_i(p)=\psi_i(|p|)$ for $i=1,2$ and $p \in \R^n$,
where $\psi,\psi_1,\psi_2$ satisfy \eqref{psi-con}--\eqref{psi-1-2}.
Let $V \in C(\T^n)$ be a potential energy with $\min_{\T^n} V=0$.

Assume that $\ol{H}$ is the effective Hamiltonian corresponding to $H(p)-V(x)$.
Assume also that $\ol{H}_i$ is the effective Hamiltonian corresponding to $H_i(p)- V(x)$ for $i=1,2$.
Then
\[
\ol{H} = \max\left\{\ol{H}_1, \ol{H}_2,0\right\}.
\]
\end{cor}

\begin{rem}
A special case of Corollary \ref{cor:rep1} is when
\[
H(p)=\psi(|p|)= \left(|p|^2-1\right)^2 \quad \text{ for } p \in \R^n,
\]
which was studied first by Armstrong, Tran and Yu \cite{ATY1}.
The method here is much simpler and more robust than that in \cite{ATY1}.
\end{rem}

By using Corollary \ref{cor:rep1} and approximation, we get another representation formula for $\ol{H}$ which will be used later.

\begin{cor}\label{cor:rep2}
Assume that \eqref{psi-con}--\eqref{psi-1-2} hold.
Set
\[
\tilde \psi_1(r)=\max\{\psi_1,0\}=
\begin{cases}
0 \qquad &\text{ for } 0\leq r \leq 1,\\
\psi(r) \qquad &\text{ for } r>1.
\end{cases}
\]
Let $H(p)=\psi(|p|)$, $\tilde H_1(p)=\tilde \psi_1(p|)$ and $H_2(p)=\psi_2(|p|)$ for  $p \in \R^n$.
Let $V \in C(\T^n)$ be a potential energy with $\min_{\T^n} V=0$.

Assume that $\ol{H},\ol{\tilde H}_1, \ol{H}_2$ are the effective Hamiltonian corresponding to $H(p)-V(x), \tilde H_1(p)- V(x), H_2(p)-V(x)$, respectively.
Then
\[
\ol{H} = \max\left\{\ol{\tilde H}_1, \ol{H}_2\right\}.
\]
\end{cor}
\noindent See Figure \ref{fig2} for the graphs of $\psi, \tilde \psi_1, \psi_2$.
\begin{figure}[h]
\begin{center}
\begin{tikzpicture}

\draw[->] (-0.5,0)--(4,0);
\draw[->] (0,-2)--(0,4);
\draw (4,-0.2) node {$r$};
\draw (1,-0.2) node {$1$};

\draw[dashed, blue, ultra thick] (0,0)--(1,0);
\draw[blue] (0.2,-0.3) node{$\tilde \psi_1$};

\draw[dashed, red, ultra thick] plot [smooth] coordinates {(1,0) (2,-0.7) (3,-1.8)};
\draw[red] (3.2, -1.9) node{$\psi_2$};

\draw[thick] plot [smooth] coordinates {(0,1) (1,0) (2,1) (3,4)};
\draw (3.2,4) node {$\psi$};

\end{tikzpicture}
\caption{Graphs of $\psi, \tilde \psi_1, \psi_2$}
\label{fig2}
\end{center}
\end{figure}
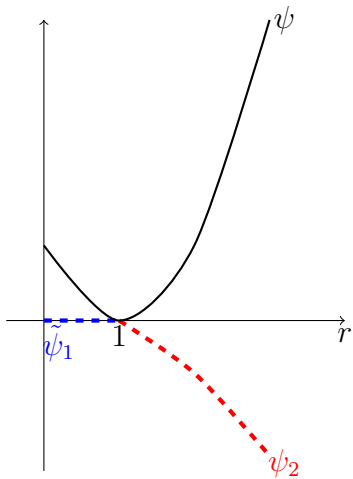

When the oscillation of $V$ is large enough, we have furthermore the following result.

\begin{cor}\label{cor:rep3}   
Let $H\in C(\R^n)$ be a coercive Hamiltonian satisfying {\rm (H1)--(H3)}, 
except that we do not require $H_2$ to be quasiconcave.   
Assume that  
\[
{\rm osc_{\T^n}}V=\max_{\T^n}V-\min_{\T^n}V\geq  \max_{\ol{U}}H=\max_{\R ^n}H_2.
\] 
Then
\[
\ol H=\max\left\{ \ol H_1,  \  -\min_{\T^n}V \right\}.
\]
In particular,  $\ol H$ is quasiconvex in this situation. 
\end{cor}

It is worth noting that the result of Corollary \ref{cor:rep3} is  interesting in the sense that
we do not require any structure of $H$ in $U$ except that $H>0$ there.

\begin{proof}
 Without loss of generality,  we assume that $\min_{\T^n}V=0$.  
 Choose a quasiconcave function $H_{2}^{+} \in C(\R^n)$ such that 
\[
\begin{cases}
\{H=0\}=\{H_{2}^{+}=0\}=\partial U,\\
 H\leq H_{2}^{+}   \text{ in } U, \text{ and }   \max_{\ol{U}}H=\max_{\R ^n}H_{2}^{+},\\
 \lim_{|p| \to \infty} H_2^+(p)=-\infty.
\end{cases}
\]
Denote $H^{+} \in C(\R^n)$ as
\[
H^{+}(p)=\max\{H, H_{2}^{+}\}=
\begin{cases}
H_1(p)  \quad &\text{for $p\in \R^n\backslash  U$,}\\
H_{2}^{+}(p)   \quad &\text{for $p\in \ol{U}$}.
\end{cases}
\]
Also denote by $\ol H^{+}$ and $\ol H_{2}^{+}$ the effective Hamiltonians associated with $H^+(p)-V(x)$ and $H_{2}^{+}(p)-V(x)$,  respectively.  
Apparently,
\begin{equation}\label{rep3-1}
\max\left\{ \ol H_1,  \  0\right\} \leq \ol H\leq \ol H^{+}.
\end{equation}
On the other hand,  by Theorem  \ref{thm:rep1},  the representation formula for $\ol{H}^+$ is
\begin{equation}\label{rep3-2}
\ol H^{+}=\max\left\{ \ol H_1,\  \ol H_{2}^{+},  \  0\right\}=\max\{ \ol H_1,  \  0 \},
\end{equation}
where the second equality is due to
\[
\ol H_{2}^{+}\leq \max_{\R ^n}H_{2}^{+}-\max_{\T ^n}V=\max_{\bar U}H-\max_{\R ^n}V\leq 0.
\]
We combine \eqref{rep3-1} and \eqref{rep3-2} to get the conclusion.
\end{proof}

\subsection{A more general case}
We first  extend Theorem \ref{thm:rep1} as following.
To avoid unnecessary technicalities, we only consider radially symmetric cases from now on.
The results still hold true for general Hamiltonians (without the radially symmetric assumption) 
under corresponding appropriate conditions.
\smallskip

Let $H:\R^n \to \R$ be such that
\begin{itemize}
\item[(H4)] $H(p)=\varphi(|p|)$ for $p \in \R^n$, where $\varphi \in C([0,\infty),\R)$ such that
\[
\begin{cases}
\varphi(0)>0, \ \varphi(2)=0, \ \lim_{r \to \infty} \varphi(r)=+\infty,\\
\varphi \text{ is strictly increasing on $[0,1]$ and $[2,\infty)$, and is strictly decreasing on $[1,2]$}.
\end{cases}
\]
\item[(H5)] $H_i(p)=\varphi_i(|p|)$ for $p\in \R^n$ and $1 \leq i \leq 3$, where $\varphi_i \in C([0,\infty),\R)$ such that
\[
\begin{cases}
\varphi_1= \varphi \text{ on } [2,\infty), \ \varphi_1 \text{ is strictly increasing on } [0,2],\\
\varphi_2 = \varphi \text{ on } [1,2], \ \varphi_2 \text{ is strictly decreasing on  $[0,1]$ and $[2,\infty)$}, \ \lim_{r \to \infty} \varphi_2(r) = -\infty,\\
\varphi_3=\varphi  \text{ on } [0,1], \ \varphi_3 \text{ is strictly increasing on } [1,\infty), \text{ and }  \varphi_3 > \varphi \text{ in } (1,\infty).
\end{cases}
\]
\end{itemize}

\begin{figure}[h]
\begin{center}
\begin{tikzpicture}

\draw[->] (-0.5,0)--(4,0);
\draw[->] (0,-2)--(0,4);
\draw (4,-0.2) node {$r$};
\draw (1,-0.2) node {$1$};
\draw (2,-0.2) node {$2$};

\draw[dashed, red, ultra thick] plot [smooth] coordinates {(0.95,2) (1.5, 2.5) (2,3)};
\draw[red] (2.2, 3) node{$\varphi_3$};

\draw[dashed, blue, ultra thick] (0,2.5)--(0.95,2);
\draw[dashed, blue, ultra thick] (2,0)--(3,-1.2);
\draw[blue] (0.2,2.7) node{$\varphi_2$};

\draw[dashed, purple, ultra thick] (0,-1)--(2,0);
\draw[purple] (0.2,-1.2) node{$\varphi_1$};

\draw[dashed] (0,2)--(1,2);
\draw[dashed] (1,0)--(1,2);

\draw[thick] plot [smooth] coordinates {(0,0.8) (1,2)  (2,0) (3,3)};
\draw (3.2,3) node {$\varphi$};

\end{tikzpicture}
\caption{Graphs of $\varphi, \varphi_1, \varphi_2, \varphi_3$}
\label{fig3}
\end{center}
\end{figure}
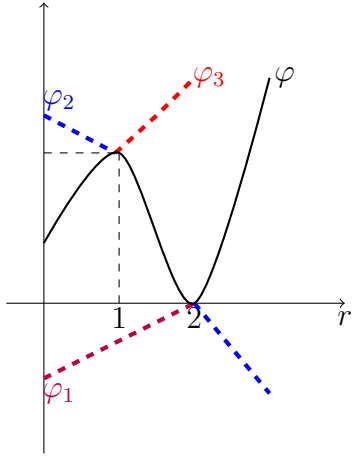

\begin{lem} \label{thm:rep2}
Let $H(p)=\varphi(|p|)$, $H_i(p)=\varphi_i(|p|)$ for $1 \leq i \leq 3$ and $p \in \R^n$,
where $\varphi,\varphi_1,\varphi_2, \varphi_3$ satisfy {\rm (H4)--(H5)}.
Let $V \in C(\T^n)$ be a potential energy with $\min_{\T^n} V=0$.

Assume that $\ol{H}$ is the effective Hamiltonian corresponding to $H(p)-V(x)$.
Assume also that $\ol{H}_i$ is the effective Hamiltonian corresponding to $H_i(p)- V(x)$ for $1 \leq i \leq 3$.
Then
\[
\begin{array}{ll}
\ol{H} &=\max\left\{0, \ol{H}_1, \ol{K}\right\}\\[5mm]
&= \max\left\{0, \ol{H}_1, \min\left\{\ol{H}_2, \ol{H}_3, \varphi(1)-\max_{\T^n} V \right\} \right\}.
\end{array}
\]
Here $\ol{K}$ is the effective Hamiltonian corresponding to $K(p)- V(x)$ for $K:\R^n \to \R$ defined as
\[
K(p)=\min\{\varphi_2(|p|), \varphi_3(|p|)\}=
\begin{cases}
\varphi(|p|) \qquad &\text{for } |p| \leq 2,\\
\varphi_2(|p|) \qquad &\text{for } |p| \geq 2.
\end{cases}
\]
In particular, both $\ol H$ and $\ol K$ are even. 

\end{lem}

\begin{proof} Considering $-K(-p)$,  thanks to the representation formula and evenness from Theorem \ref{thm:rep1},
\[
\ol{K}=\min\left\{\ol{H}_2, \ol{H}_3, \varphi(1)-\max_{\T^n} V \right\}.
\]
Define $\tilde \varphi_2 = \min\left\{\varphi_2, \varphi(1)\right\}$.
Let $\tilde H_2(p) = \tilde \varphi_2(|p|)$ and $\ol{\tilde H}_2$ be the effective Hamiltonian corresponding to $\tilde H_2(p)- V(x)$.
Then, thanks to Corollary \ref{cor:rep2}, we also have that
\begin{equation}\label{g0-1}
\ol{K} = \min \left\{ \ol{\tilde H}_2, \ol{H}_3 \right\}.
\end{equation}
Our goal is then to show that $\ol{H} = \max\left\{0, \ol{H}_1, \ol{K}\right\}$.
To do this, we again divide the proof into few steps for clarity. 
{\it Readers should notice that the proof below does not depend on  the quasiconvexity of $\ol{H}_3$.  
It only uses the fact that  $\ol{H}_3\geq \ol{H}$. 
This is essential to prove the most general result, Theorem \ref{Maintheorem}.}
\smallskip

\noindent {\bf Step 1.} Clearly $0 \leq \ol{H} \leq H$. This implies further that
\begin{equation} \label{g1-1}
\ol{H}(p)=0 \quad \text{ for all } |p|=2.
\end{equation}
We furthermore have that $\ol{K}, \ol{H}_1 \leq \ol{H}$ as $K, H_1 \leq H$. Thus,
\begin{equation}\label{g1-2}
\ol{H} \geq \max\left\{0, \ol{H}_1, \ol{K}\right\}
\end{equation}
We now show the reverse inequality of \eqref{g1-2} to finish the proof.
\smallskip

\noindent {\bf Step 2.} Fix $p \in \R^n$.
Assume that $\ol{H}_1(p) \geq \max\left\{0, \ol{K}(p) \right\}$.  Since $H_1$ is quasiconvex, 
we follow exactly the same lines of Step 2 in the proof of Theorem \ref{thm:rep1}
to deduce that $\ol{H}_1(p) \geq \ol{H}(p)$.
\smallskip

\noindent {\bf Step 3.} Assume that $\ol{K}(p) \geq \max\left\{0, \ol{H}_1(p) \right\}$.  Since $K$ is not quasiconvex or quasiconcave,  we cannot directly copy Step 2 or Step 3 in the proof of Theorem \ref{thm:rep1}.   Instead,  there are two cases that need to be considered.

Firstly, we consider the case that $\ol{K}(p) = \ol{\tilde H}_2(p) \leq \ol{H}_3(p)$.
Let $v(x,p)$ be a solution to the cell problem
\begin{equation}\label{g3-1}
\tilde H_2(p+Dv(x,p)) - V(x)=\ol{\tilde H}_2(p) \geq 0 \quad \text{ in } \T^n.
\end{equation}
Since $\tilde H_2$ is quasiconcave, for any $x \in \T^n$ and $q \in D^+v(x,p)$, we have
\[
\tilde H_2(p+q) - V(x)= \ol{\tilde H}_2(p) \geq 0,
\]
which gives that $\tilde H_2(p+q) \geq 0$ and hence $\tilde H_2(p+q) \geq H(p+q)$.
Therefore, $v(x,p)$ is a viscosity subsolution to
\[
H(p+Dv(x,p)) - V(x) = \ol{\tilde H}_2(p) \quad \text{ in } \T^n.
\]
We conclude that $\ol{K}(p) = \ol{\tilde H}_2(p) \geq \ol{H}(p)$.

Secondly, assume that $\ol{K}(p) = \ol{H}_3(p) \leq  \ol{\tilde H}_2(p)$. Since $\varphi_3\geq \varphi$,   $\ol{H}_3(p)\geq  \ol{H}(p)$.   
Combining with $ \ol{H}(p)\geq  \ol{K}(p)$ in \eqref{g1-2}, we obtain  $\ol{K}(p)= \ol{H}(p)$ in this step.  
\smallskip

\noindent {\bf Step 4.}
Assume that $0> \max\left\{\ol{H}_1(p), \ol{K}(p)\right\}$.
Our goal now is to show $\ol{H}(p)=0$.
Thanks to \eqref{g1-1} in Step 1, we may assume that $|p| \neq 2$.

For $\sig \in [0,1]$ and $i=1,2$, let $\ol{H}^\sig, \ol{H}_1^\sig$, $\ol{K}^\sig$ be the effective Hamiltonians corresponding to $H(p)-\sig V(x), H_1(p) - \sig V(x)$, $K(p)-\sig V(x)$, respectively. It is clear that
\begin{equation}\label{g4-1}
0 \leq \ol{H}^1=\ol{H} \leq \ol{H}^\sig \quad \text{ for all } \sig \in [0,1].
\end{equation}
By repeating Steps 2 and 3 above, we get 
\begin{equation} \label{g4-2}
\text{For $p\in\R^n$ and $\sig \in [0,1]$, if $\max\left\{\ol{H}^\sig_1(p),\ol{K}^\sig(p)\right\}=0$, then $\ol{H}^\sig(p)=0$.}
\end{equation}
We only consider the case $|p|<2$ here.
The case $|p|>2$ is analogous. Notice that
\[
H(p)= K(p) = \ol{K}^0(p)>0 \quad \text{and} \quad \ol{K}(p)= \ol{K}^1(p) <0.
\]
By the continuity of $\sig \mapsto \ol{K}^\sig(p)$, there exists $s \in (0,1)$ such that $\ol{K}^s(p)=0$.
Note furthermore that, as $|p|<2$, $\ol{H}_1^s(p) \leq H_1(p) <0$.
These, together with \eqref{g4-1} and \eqref{g4-2}, yield the desired result.
\end{proof}

\subsection{General cases}\label{generalcase}
By using induction, 
we can obtain min-max (max-min) formulas for $\ol{H}$ in case $H(p)=\varphi(|p|)$ where $\varphi$ satisfies some certain conditions described below.
We consider two cases corresponding to Figures \ref{fig4} and \ref{fig5}.
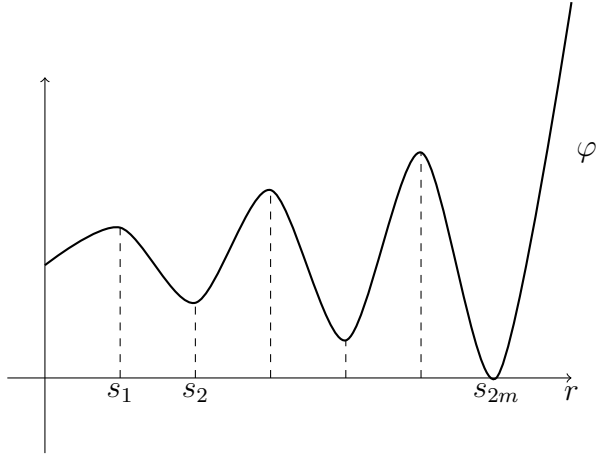
\begin{figure}[h] 
\begin{center}
\begin{tikzpicture}

\draw[->] (-0.5,0)--(7,0);
\draw[->] (0,-1)--(0,4);
\draw (7,-0.2) node {$r$};

\draw[thick] plot [smooth] coordinates {(0,1.5) (1,2)  (2,1) (3,2.5) (4,0.5) (5,3) (6,0) (7,5)};
\draw (7.2,3) node {$\varphi$};

\draw[dashed] (1,0)--(1,2);
\draw[dashed] (2,0)--(2,1);
\draw[dashed] (3,0)--(3,2.5);
\draw[dashed] (4,0)--(4,0.5);
\draw[dashed] (5,0)--(5,3);

\draw (1,-0.2) node {$s_1$};
\draw (2,-0.2) node {$s_2$};
\draw (6,-0.2) node {$s_{2m}$};

\end{tikzpicture}

\caption{Graph of $\varphi$ in first general case}
\label{fig4}
\end{center}
\end{figure}

In this first general case corresponding to Figure \ref{fig4}, we assume that
\begin{itemize}
\item[(H6)] $\varphi \in C([0,\infty),\R)$ satisfying 
\[
\begin{cases}
\text{there exist $m \in \N$ and $0=s_0 <s_1 < \ldots s_{2m}<\infty=s_{2m+1}$ such that}\\
\text{$\varphi$ is strictly increasing in $(s_{2i},s_{2i+1})$, and is strictly decreasing in $(s_{2i+1},s_{2i+2})$,}\\
\text{$\varphi(s_0)>\varphi(s_2)>\ldots > \varphi(s_{2m})$, and  $\varphi(s_1) < \varphi(s_3) <\ldots < \varphi(s_{2m+1})=\infty$.}
\end{cases}
\]

\end{itemize}

\noindent For $0\leq i\leq m$, 

$\bullet$ let $\varphi_{2i}:[0,\infty)\to \mathbb{R}$ be a continuous, strictly increasing function such that $\varphi_{2i}=\varphi$ on $[s_{2i},s_{2i+1}]$ and  $\lim_{s\to \infty} \varphi_{2i}(s)=\infty$.  Also  $\varphi_{2i}\geq \varphi_{2i+2}$. 

$\bullet$  let $\varphi_{2i+1} : [0,\infty)\to \mathbb{R}$ be a continuous, strictly decreasing function such that $\varphi_{2i+1}=\varphi$ on $[s_{2i+1}, s_{2i+2}]$ and $\lim_{s\to \infty} \varphi_{2i+1}(s)=-\infty$.   Also  $\varphi_{2i+1}\leq \varphi_{2i+3}$.

Define 
\[
H_{m-1}(p)=\max\{\varphi(|p|), \ \varphi_{2m-2}(|p|)\}=
\begin{cases}
\varphi(|p|)  \quad &\text{for $|p|\leq s_{2m-1}$,}\\
\varphi_{2m-2}(|p|) \quad &\text{for $|p|>s_{2m-1}$}
\end{cases}
\]

and

\[
k_{m-1}(s)=\min\{\varphi(s), \ \varphi_{2m-1}(s)\}=
\begin{cases}
\varphi(s)  \quad &\text{for $s\leq s_{2m}$,}\\
\varphi_{2m-1}(s) \quad &\text{for $s>s_{2m}$.}
\end{cases}
\]

Denote $\ol{H}_{m-1}$,  $\ol H_{m}$,  $\ol{K}_{m-1}$, $\ol{\Phi}_{j}$ as the  effective Hamiltonians associated with 
the Hamiltonians $H_{m-1}(p)-V(x)$,  $\varphi (|p|) - V(x)$,  $k_{m-1}(|p|)- V(x)$ and $\varphi_{j}(|p|)- V(x)$ for $0\leq j\leq 2m$, respectively.  

\medskip

The following is our main decomposition theorem in this paper. 

\begin{thm}\label{Maintheorem}  Assume that {\rm (H6)} holds for some $m\in \N$. Then
\begin{equation}\label{mainfor-1}
 \ol{H}_m=\max\left\{\ol{K}_{m-1}, \  \ol{\Phi}_{2m},\ \varphi(s_{2m})-\min_{\T^n}V\right\},
\end{equation}
and
\begin{equation}\label{mainfor-2}
\ol{K}_{m-1}=\min\left\{\ol{H}_{m-1}, \  \ol{\Phi}_{2m-1},\ \varphi(s_{2m-1})-\max_{\T^n}V\right\}.
\end{equation}
In particular,  $ \ol{H}_m$ and $ \ol{K}_{m-1}$ are both even. 
\end{thm}

Again,  we would like to point out that the evenness of $ \ol{H}_m$ and $ \ol{K}_{m-1}$ is far from being obvious although $H_m$ and $K_m$ are both even. 
See the discussion in Subsection \ref{subsec:even} for this subtle issue. 

\begin{proof}
We prove by induction. 
When $m=1$, the two formulas  \eqref{mainfor-1} and \eqref{mainfor-2}  follow from Lemma \ref{thm:rep2} and Theorem \ref{thm:rep1}.

Assume that \eqref{mainfor-1} and \eqref{mainfor-2} hold for $m \in \N$.  We need to verify these equalities for $m+1$. 
Using similar arguments as those in the proof Lemma \ref{thm:rep2}, noting  the statement in italic right above Step 1, we first derive that 
\[
\ol{K}_{m}=\min\left\{\ol{H}_{m}, \  \ol{\Phi}_{2m+1},\ \varphi(s_{2m+1})-\max_{\T^n}V\right\}.
\]
Then again, by basically repeating the proof of  Lemma \ref{thm:rep2}, we obtain
\[
 \ol{H}_{m+1}=\max\left\{\ol{K}_{m}, \  \ol{\Phi}_{2m+2},\ \varphi(s_{2m+2})-\min_{\T^n}V\right\}.
\]
\end{proof}

\begin{rem}\label{rem:flat}  
(i) By approximation, we see that representation formulas \eqref{mainfor-1} and \eqref{mainfor-2} still hold true if we relax (H6) a bit, that is, we only require that
$\varphi$ satisfies
\[
\begin{cases}
\text{$\varphi$ is  increasing in $(s_{2i},s_{2i+1})$, and is decreasing in $(s_{2i+1},s_{2i+2})$,}\\
\text{$\varphi(s_0) \geq \varphi(s_2) \geq \ldots \geq  \varphi(s_{2m})$, and $\varphi(s_1) \leq \varphi(s_3) \leq \ldots < \varphi(s_{2m+1})=\infty$.}
\end{cases}
\]

(ii) According to Corollary  \ref{cor:rep3},  if ${\rm osc}_{\T^n}V =\max_{\T ^n}V-\min_{\T^n}V\geq \varphi(s_{2m-1})-\varphi(s_{2m})$, then $\ol H$ is quasiconvex and 
\[
\ol H_m=\max\left\{\ol \Phi_{2m}, \ \varphi(s_{2m})-\min_{\T^n}V\right\}.
\]

\end{rem}

The second general case corresponds to the case where  $H(p)=-k(|p|)$ for all $p \in \R^n$ as described in Figure 5.
After changing the notations appropriately, we obtain similar representation formulas as in Theorem \ref{Maintheorem}.
We omit the details here.
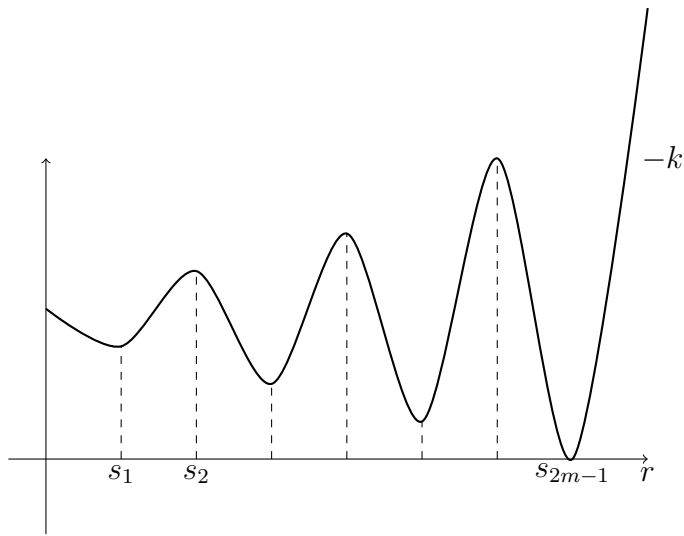
\begin{figure}[h]
\begin{center}
\begin{tikzpicture}

\draw[->] (-0.5,0)--(8,0);
\draw[->] (0,-1)--(0,4);
\draw (8,-0.2) node {$r$};

\draw[thick] plot [smooth] coordinates {(0,2) (1,1.5)  (2,2.5) (3,1) (4,3) (5,0.5) (6,4) (7,0) (8,6)};

\draw[dashed] (1,0)--(1,1.5);
\draw[dashed] (2,0)--(2,2.5);
\draw[dashed] (3,0)--(3,1);
\draw[dashed] (4,0)--(4,3);
\draw[dashed] (5,0)--(5,0.5);
\draw[dashed] (6,0)--(6,4);

\draw (1,-0.2) node {$s_1$};
\draw (2,-0.2) node {$s_2$};
\draw (7,-0.2) node {$s_{2m-1}$};
\draw (8.2,4) node {$-k$};
\end{tikzpicture}

\caption{Graph of $-k$ in the second general case}
\label{fig5}
\end{center}
\end{figure}

\subsection{ ``Non-decomposability" and  Breakdown of symmetry}\label{subsec:even}

A natural question is whether we can extend Theorem \ref{Maintheorem} to other  nonconvex $H$,   i.e.,  there exist quasiconvex/concave $H_i$  ($1\leq i\leq m$) such that $\ol H$ is given by a ``decomposition" formula (e.g., min-max type)  involving $\ol H_i$, $\min V$ and $\max V$
\begin{equation}\label{uni-decom}
\ol H=G(\ol H_1,..., \ol H_m, \ \min V, \ \max V),
\end{equation}
 for any $V\in C(\Bbb T^n)$.  Here $\ol H$ and $\ol H_i$ are effective Hamiltonians associated with $H-V$ and $H_i-V$.  Note that for quasiconvex/concave function $F$,   using the inf-max formula (\ref{convex-min-max}),  it is easy to see that the effective Hamiltonians associated with $F(p)-V(x)$ and $F(p)-V(-x)$ are the same.   Hence if such a ``decomposition"  formula indeed exists for a specific nonconvex $H$,   effective Hamiltonians associated with $H(p)-V(x)$ and $H(p)-V(-x)$ have to be identical as well.  In particular,  if $H$ is an even function,  this is equivalent to saying that $\ol H$ is even too, which leads to the following  question. 

\begin{quest}\label{quest:even}
Let $H \in C(\R ^n)$ be a coercive and even Hamiltonian,  and $V\in C(\T^n)$ be a given potential.
Let $\ol{H}$ be the effective Hamiltonian associated with $H(p)-V(x)$.
Is it true that  $\ol H$  is also even?  In general,  we may ask  what  properties of the original Hamiltonian  will  be preserved under homogenization.  

\end{quest}

Even though that this is a simple and natural question, it has not been studied much in the literature as far as the authors know.
We give below some answers and discussions to this:

\smallskip

$\bullet$  If $H$ is quasiconvex,  the answer is ``yes"  due to the inf-max formula
\[
\ol{H}(p)= \inf_{\phi \in C^1(\T^n)} \max_{x \in \R^n} \left ( H(p+D\phi(x)) - V(x) \right)
\]
as shown in the proof of Lemma \ref{thm:rep1}.

$\bullet$  For  genuinely nonconvex $H$,  if  $\ol{H}$ can be written as a min-max formula involving effective Hamiltonians of even quasiconvex (or quasiconcave) Hamiltonians,
then  $\ol H$ is still even (e.g.,  see Corollary \ref{cor:rep1}, Lemma \ref{thm:rep2}, and Theorem \ref{Maintheorem}). 

$\bullet$  However, in general, the evenness is lost as presented in  Remark 1.2 in  \cite{LTY}.
Let us quickly recall the setting there.
We consider the case $n=1$, and choose  $H(p)=\varphi(|p|)$ for $p\in\R$, where $\varphi$ satisfies (H8)  (see Figure \ref{fig6} below)
with $m_1=\frac{1}{3}$ and $M_1=\frac{1}{2}$ .
Fix $s\in (0,1)$, and set $V_s(x)=\min\left\{{x\over s},\  {1-x\over 1-s}\right\}$ for $x\in  [0,1]$. 
Extend $V$ to $\R$ in a periodic way. 
Then  $\ol H$ is not even unless $s={1\over 2}$. 
In particular, this implies that a decomposition formula for $\ol H$ does not exist.  Also, see Figure \ref{qtyEx1_5_weno3d} below for loss of evenness when the Hamiltonian is of double-well type.

$\bullet$  It is extremely interesting if we can point out some further general requirements on $H$ and $V$
in the genuinely nonconvex setting, under which $\ol{H}$ is  even.
The interplay between $H$ and $V$ plays a crucial role here (see Remark \ref{rem:conj} for intriguing observations). 

Some related  discussions  and  interesting applications of evenness can also be found in  \cite{See}.

\section{Quasi-convexification phenomenon in multi-dimensional spaces} \label{sec:quasi}

Intuitively,  homogenization, a nonlinear averaging procedure, makes the effective Hamiltonian less nonconvex. 
The question is how to describe this in a rigorous and systematic way. 
Some special cases have been handled  in Remark  \ref{rem:flat}. In this section,  we look at more generic and  important situations: general radially symmetric Hamiltonians and a typical double-well Hamiltonian.  These two types of Hamiltonians more or less capture essential features of nonconvexity.    In some sense,  quasi-convexification represents  a scenario where there is no genuine decomposition of $\ol H$.  
Due to the difficulty in rigorous analysis,  we focus more on numerical computations.  
The Lax-Friedrichs based big-T method is used  to compute the effective Hamiltonian.

 \subsection{Radially symmetric Hamiltonians} 
Assume that $H(p)=\varphi(|p|)$ for all $p\in \R^n$, where $\varphi:[0,\infty)\to  \mathbb{R}$ is a given function.
 The following is quite a general condition on $\varphi$.
 
 \begin{itemize}
\item[(H7)] $\varphi \in C([0,\infty),\R)$ satisfying that $\lim_{s\to \infty} \varphi(s)=+\infty$ and 
\[
\begin{cases}
\text{there exist $m \in \N$ and $0=s_0 <s_1 < \ldots s_{2m}<\infty=s_{2m+1}$ such that}\\
\text{$\varphi$ is strictly increasing in $(s_{2i},s_{2i+1})$, and is strictly decreasing in $(s_{2i+1},s_{2i+2})$.}
\end{cases}
\]
\end{itemize}
 It is clear that (H7) is more general than (H6).  In fact,  any coercive function $\psi \in C([0,\infty),\R)$ can be approximated by $\varphi$ satisfying (H7). 

Denote by
\[
\begin{cases}
M_i=\varphi(s_{2i-1})   \quad &\text{ for } 1\leq i \leq  m,\\
m_j=\varphi(s_{2j})  \quad &\text{ for } 1 \leq j \leq m.
\end{cases}
\]

We propose the following conjecture.  

\begin{conj}\label{conj:convex} 
Assume that {\rm (H7)} holds. 
Assume further that $\varphi(0)=\min \varphi=0$.
Let $H(p)=\varphi(|p|)$ for all $p\in \R^n$, and $V\in C(\T^n)$ be a given potential function.
Let $\ol{H}$ be the effective Hamiltonian corresponding to $H(p)-V(x)$.
If 
\[
{\rm osc}_{\T ^n}V=\max_{\T^n} V - \min_{\T^n} V \geq   \max_{i, j}(M_i-m_j),
\]  
then the effective Hamiltonian $\ol {H}$ is quasiconvex.  
\end{conj}

When $n=1$, the above conjecture was proven in  \cite{ATY2} based on some essentially one dimensional approaches.  
In multi-dimensional spaces, this conjecture seems quite challenging in general (Remark 3 is a special case).  
Let us now consider a basic situation, which we believe is an important step toward proving Conjecture \ref{conj:convex}.
\begin{itemize}
\item[(H8)] $\varphi \in C([0,\infty),\R)$ such that
there exist $0<s_1<s_2<\infty$ satisfying
\[
\begin{cases}
\text{$\varphi(0)=0 < \varphi(s_2)=m_1 < \varphi(s_1)=M_1<\lim_{s \to \infty} \varphi(s)=+\infty$,}\\
\text{$\varphi$ is strictly increasing on $[0,s_1]$ and $[s_2,\infty)$ and $\varphi$ is strictly decreasing on $[s_1,s_2]$.}
\end{cases}
\]
\end{itemize}
See Figure \ref{fig6}.  For this particular case,  the conjecture says that $\ol H$ is quasiconvex if 
$$
{\rm osc}_{\T ^n}V\geq M_1-m_1.
$$
This is clear in terms of numerical results (Figure \ref{qtyEx2_1}).  

\smallskip

\noindent{\bf Numerical example 1.}
Let $n=2$.  We consider the following setting
\begin{align*}
&H(p) = \min\left\{4\sqrt{p_1^2+p_2^2},\; 2\left|\sqrt{p_1^2+p_2^2}-1\right|+1\right\} \quad &\text{ for } p=(p_1,p_2)\in \R^2,  \\
&V(x)= S*(1+\sin(2 \pi x_1))( 1+ \sin(2\pi x_2))  \quad &\text{ for } x=(x_1,x_2)\in \T^2.
\end{align*}
The constant $S$ serves as the scaling parameter to increase or decrease the effect of the potential energy $V$.  
For this specific case,  $M_1-m_1=2-1=1$ and ${\rm osc}_{\T ^2}V=4S$.  So the threshold value is $S=0.25$. 

The Lax-Friedrichs based big-T method is used  to compute this effective Hamiltonian.  
The computational $x$ domain $[0,1]^2$ is discretized by $401\times 401$ mesh points and the $p$ domain $[-1,1]^2$ is sampled by $21\times 21$ mesh points. 
The initial condition for the big-T method is taken to be $\cos(2\pi x_1)\sin(2\pi x_2)$.  See Figure \ref{qtyEx2_1}.

\begin{figure}[htb]
 \centering
 (a){\includegraphics[scale=0.22]{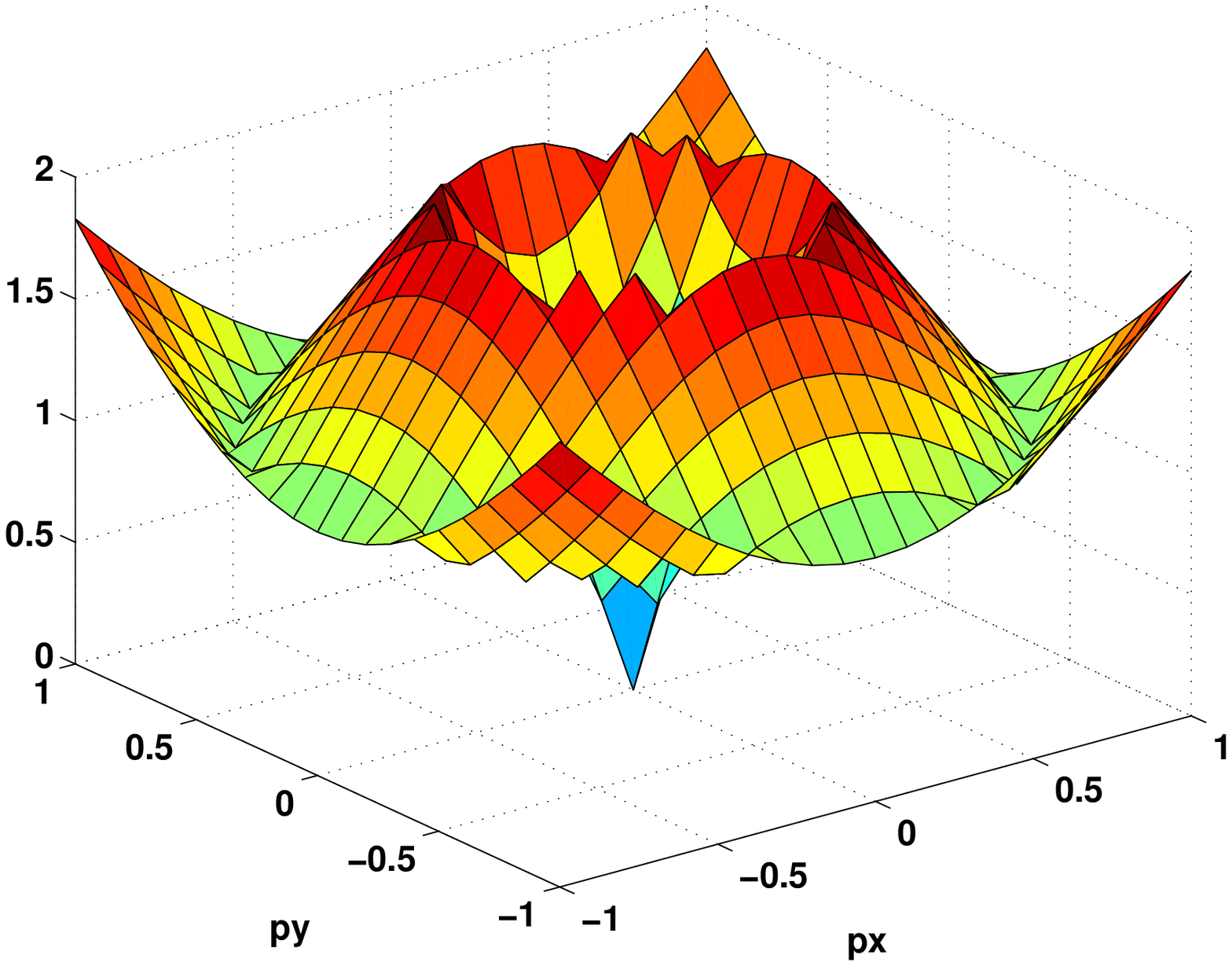}}
 (b){\includegraphics[scale=0.22]{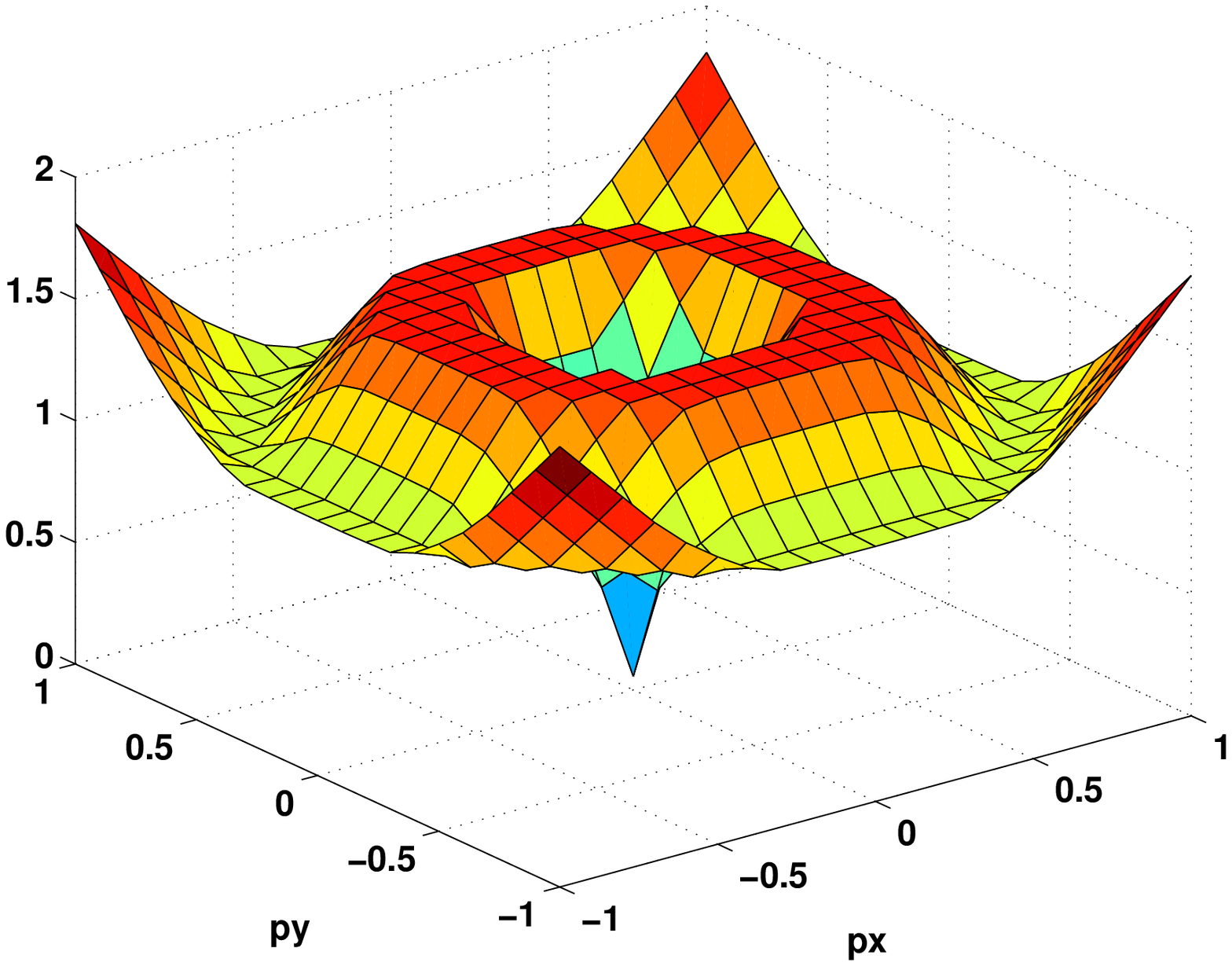}}
 (c){\includegraphics[scale=0.22]{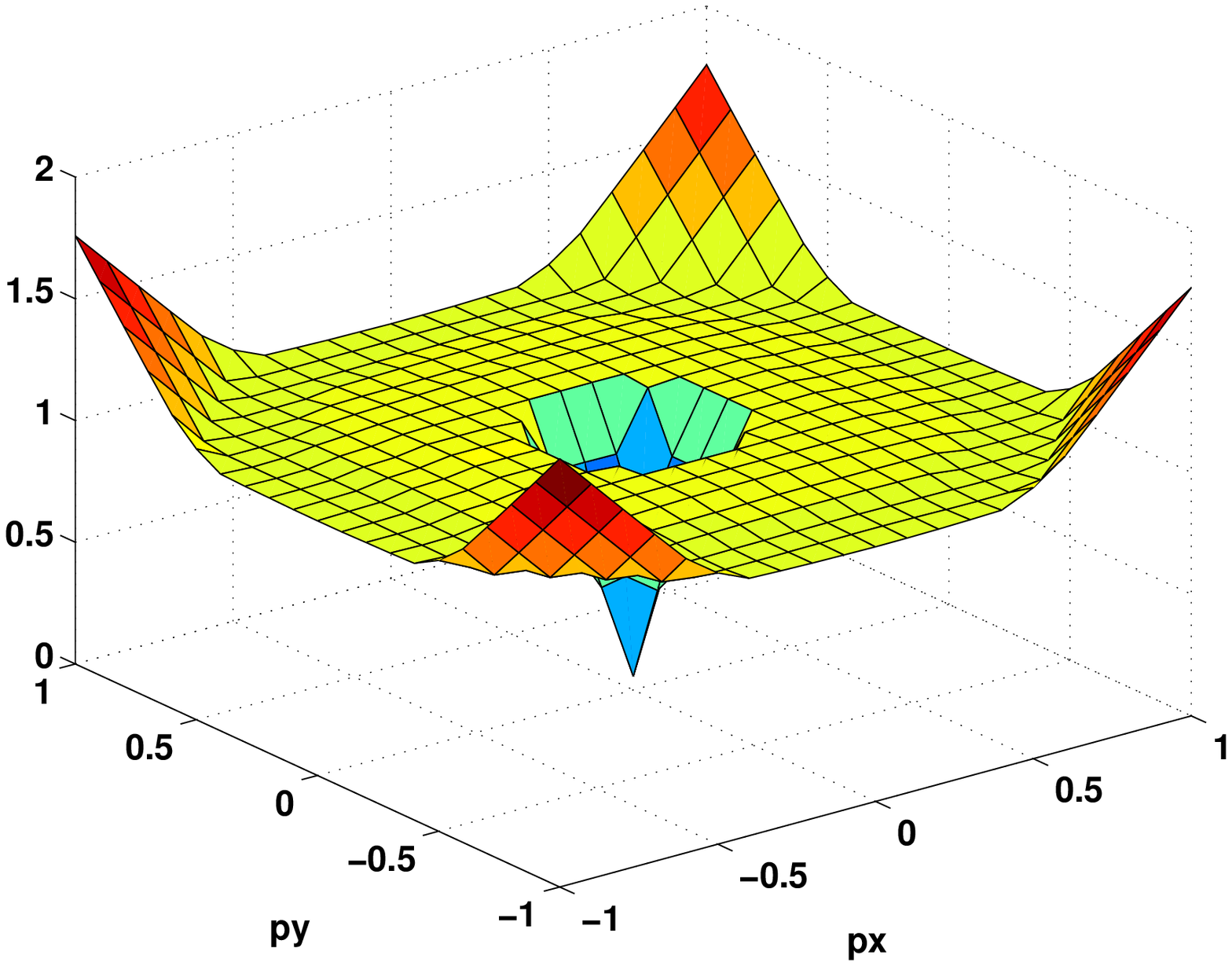}}\\
 (d){\includegraphics[scale=0.22]{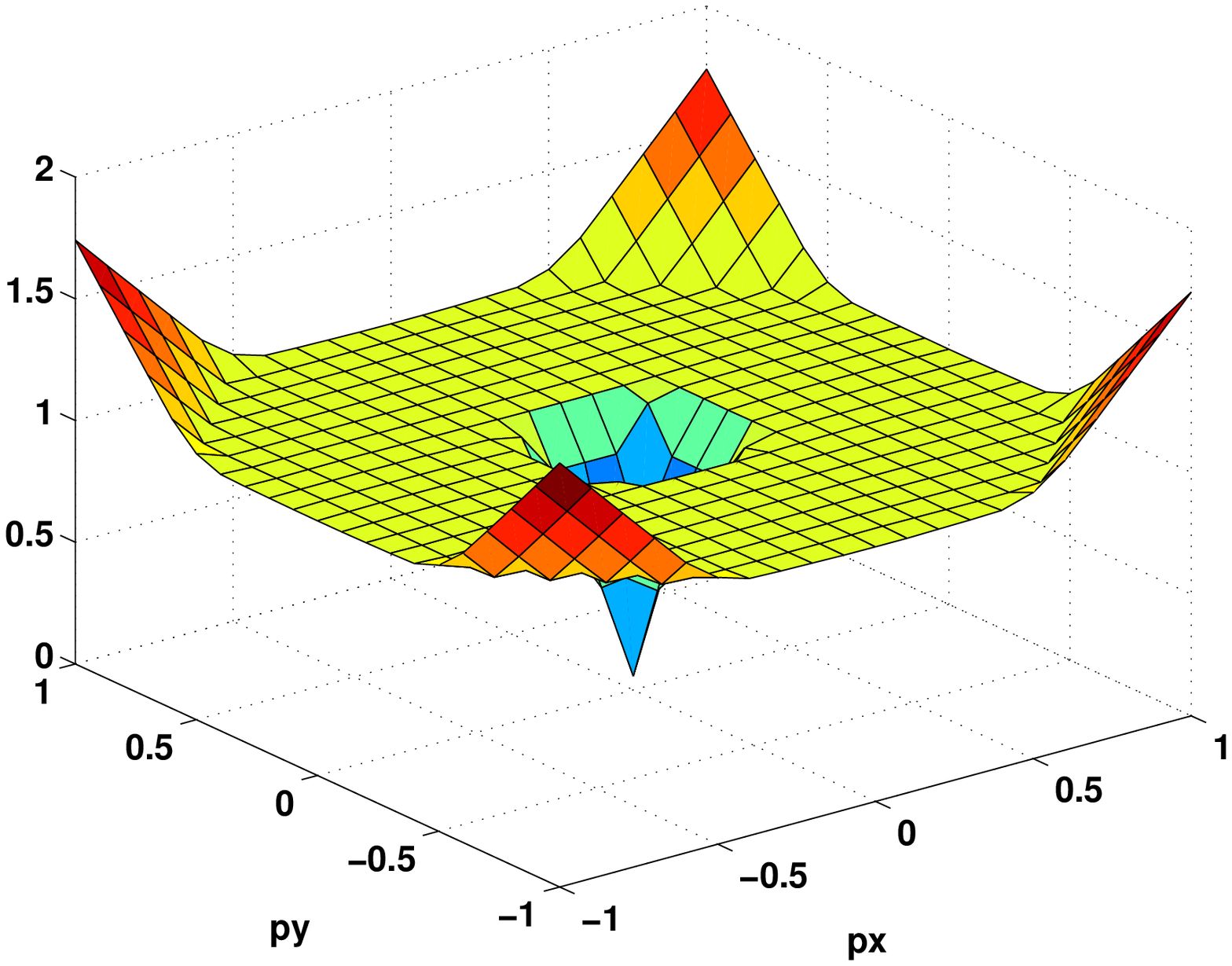}}
 (e){\includegraphics[scale=0.22]{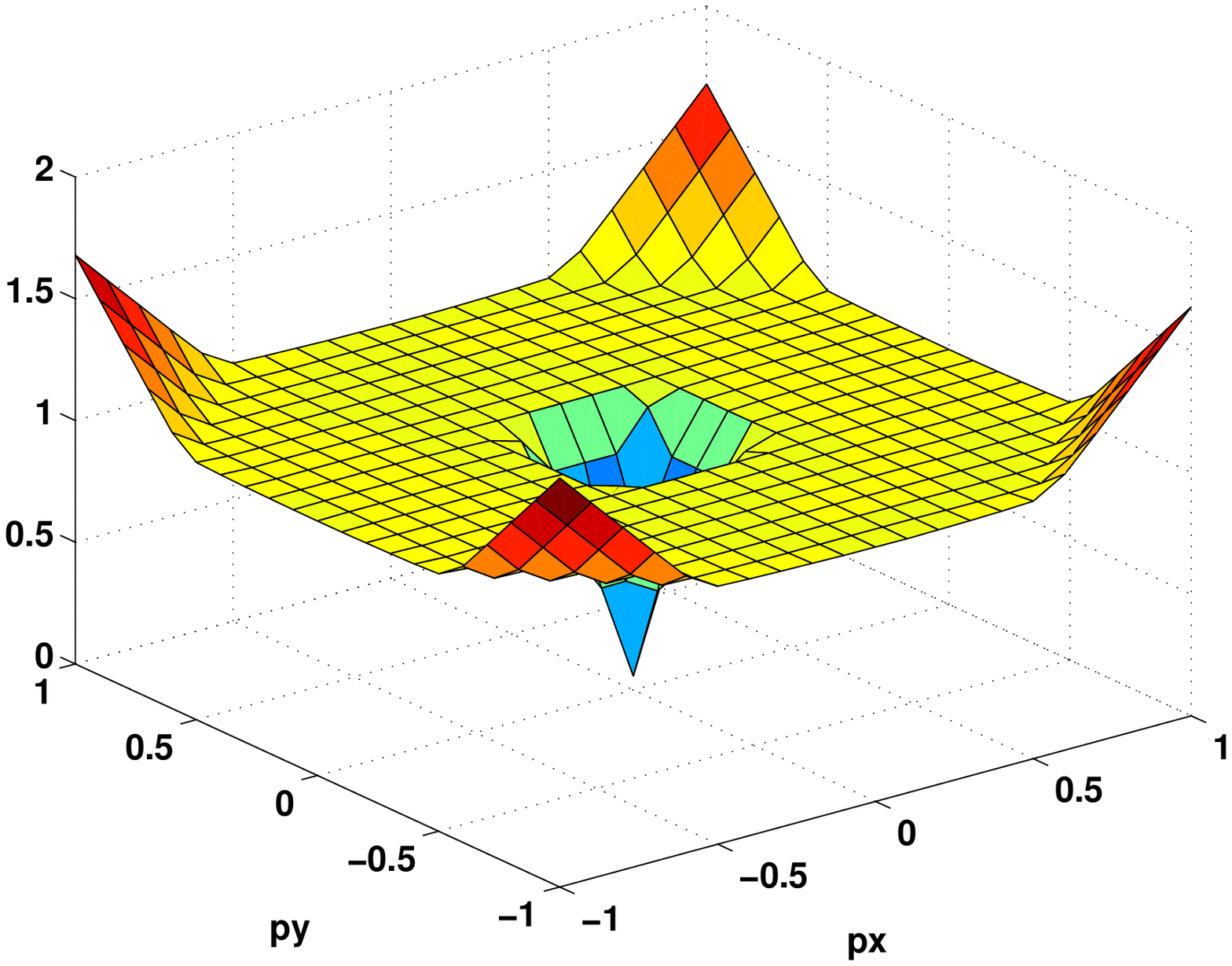}}\\
 \caption{(a) The original Hamiltonian: $S=0$. (b) $S=0.125$. (c) $S=0.25$. (d) $S=0.30$. (e) $S=0.50$.} 
\label{qtyEx2_1}
\end{figure}

However, we are only able to rigorously  verify this for level sets above $m_1$. This partially demonstrate the ``quasi-convexification" since the nonconvexity of the original $H$ appears on level sets between $m_1$ and $M_1$.   Denote $H_1=\max\{H,m_1\}$ and 
\begin{equation}\label{eq-H2}
H_2(p)=
\begin{cases}
H(p) \quad &\text{ for } |p| \leq s_1,\\
\max\{M_1,H(p)\} \quad &\text{ for } |p| \geq s_1.
\end{cases}
\end{equation}
Note that $H_1$ is a ``decomposable" nonconvex function from  Remark  \ref{rem:flat} and $H_2$ is quasiconvex.
Precisely speaking, 

\begin{thm} \label{thm:conj}
Assume that {\rm (H8)} holds. 
Let $H(p)=\varphi(|p|)$ for all $p\in \R^n$, and $V\in C(\T^n)$ be a given potential energy
such that 
$$
{\rm osc}_{\T ^n}V\geq  M_1-m_1.
$$
Let $\ol{H}$ be the effective Hamiltonian corresponding to $H(p)-V(x)$. Then for any $\mu\geq m_1$, the level set
$$
\{\ol H\leq \mu\}
$$
is convex.  
\end{thm}

\begin{figure}[h]
\begin{center}
\begin{tikzpicture}

\draw[->] (-0.5,0)--(8,0);
\draw[->] (0,-1)--(0,4);
\draw (8,-0.2) node {$r$};

\draw[dashed] (0,2)--(1,2);
\draw[dashed] (1,0)--(1,2);
\draw (1,-0.3) node{$s_1$};
\draw (-0.3,2) node{$M_1$};

\draw[dashed] (0,1)--(2,1);
\draw[dashed] (2,0)--(2,1);
\draw (2,-0.3) node{$s_2$};
\draw (-0.3,1) node{$m_1$};

\draw[thick] plot [smooth] coordinates {(0,0) (1,2)  (2,1) (3,4)};
\draw (3.2,4) node {$\varphi$};

\end{tikzpicture}
\caption{Graphs of $\varphi$}
\label{fig6}
\end{center}
\end{figure}
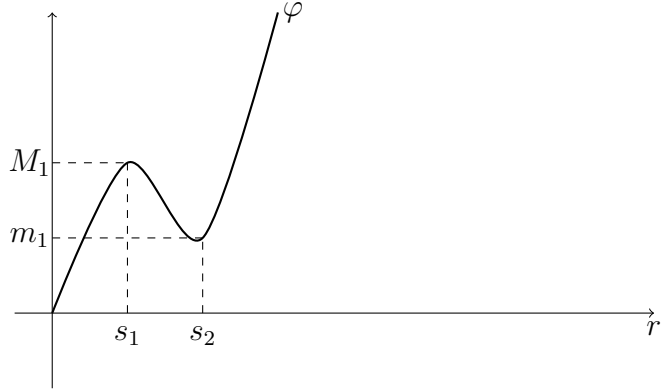

\begin{proof}
Without loss of generality,  we assume that $\min_{\T^n}V=0$. Let $\ol{H}_1$ be the effective Hamiltonian of $H_1-V$. 
Clearly $\ol{H}_1$ is quasiconvex by  Remark  \ref{rem:flat}.  So  it  suffices to show that for every $\mu>m_1$,  
$$
\ol H(p)=\mu   \quad \text{ if and only if}  \quad \ol H_1(p)=\mu.
$$
Since $\ol H\leq \ol H_1$,  we only need to show that  for fixed  $p \in \R^n$, if $\ol H_1(p)>m_1$, then $\ol H(p)=\ol H_1(p)$.   In fact,  let $v(x,p)$ be a solution to
\[
H_1(p+Dv) - V = \ol{H}_1(p)> m_1 \quad \text{ in } \T^n.
\]
Note that $\ol{H}_1(p)+V>m_1$.
It is straightforward that $v$ is also a solution to 
\[
H(p+Dv) - V = \ol{H}(p) \quad \text{ in } \T^n.
\]
Thus, $\ol{H}_1(p)=\ol{H}(p)$. The proof is complete.
\end{proof}

\begin{rem}\label{rem:conj}  Here is an interesting transition between min-max decomposition,  evenness and quasi-convexification when $n=1$. Assume $\min_{\T} V =0$.  

$\bullet$ If  $c_0=\max_{\T} V< M_1-m_1$,  it is not hard to  obtain a representation formula for $\ol{H}$ (``conditional decomposition")
\begin{equation}\label{rep-min}
\ol{H} = \min \left\{\ol{H}_1, \ol{H}_2 \right\}.
\end{equation}
Here, $\ol{H}_1$ falls into the category of item (i) in Remark \ref{rem:flat}.
And $\ol H_2$ is the effective Hamiltonian associated with the quasiconvex $H_2$ in (\ref{eq-H2}).  In particular,  $\ol H$ is {\bf even but not quasiconvex}.   The shape of $\ol H$ is qualitatively similar to that of $H$.  It is not clear to us whether this decomposition formula holds when $n\geq 2$.  The key is to answer Question 3 in the appendix first.  

\smallskip

$\bullet$  If $c_0=\max_{\T} V=M_1-m_1$,  $\ol H$ {\bf is both even and quasiconvex}. 

\smallskip

$\bullet$ If $c_0=\max_{\T} V>M_1-m_1$ and $V_s(x)=c_0\min\left\{{x\over s},  {1-x\over 1-s}\right\}$  for $x\in  [0,1]$ (extend $V$ to $\R$ periodically),  
then $\ol H$ {\bf is quasiconvex but  loses evenness}.  More precisely,  by adapting Step 1 in the proof of Theorem 1.4 in \cite{LTY},  we can show that the level set $\left\{\ol H\leq \mu\right\}$ is not even for any $\mu\in  (M_1-c_0, m_1)$ if $s\not={1\over 2}$. Hence the above  formula or decomposition \eqref{rep-min} no longer holds.

See also  Figure  \ref{qtyEx3_2} below for numerical computations of a specific example.  
\end{rem}

\noindent{\bf Numerical example 2.}  Let $n=1$. We consider the following setting
\begin{align*}
&H(p) = \min\left\{4 |p|, 2|\, |p|-1|+1\right\}, \quad &\text{ for } p \in \R,\\
&V(x)= S*\min\left\{3 x, \frac{3}{2}(1-x)\right\}, \quad &\text{ for } x \in [0,1],
\end{align*}
and extend $V$ to $\R$ in a periodic way. 
The constant $S$ serves as the scaling parameter to increase or decrease the effect of the potential energy $V$.  
For this case, $c_0=1$ and $M_1-m_1=1$.  

The Lax-Friedrichs based big-T method  is used to compute this effective Hamiltonian. 
The $p$ domain $[-1,1]$ is sampled by $41$ mesh points. 
 The computational $x$ domain $[0,1]$ is discretized by $401$ mesh points.
The initial condition for the big-T method is taken to be $\cos(2\pi x)$. 
The results are shown in Figure \ref{qtyEx3_2}.

\begin{figure}[htb]
 \centering
  (a){\includegraphics[scale=0.22]{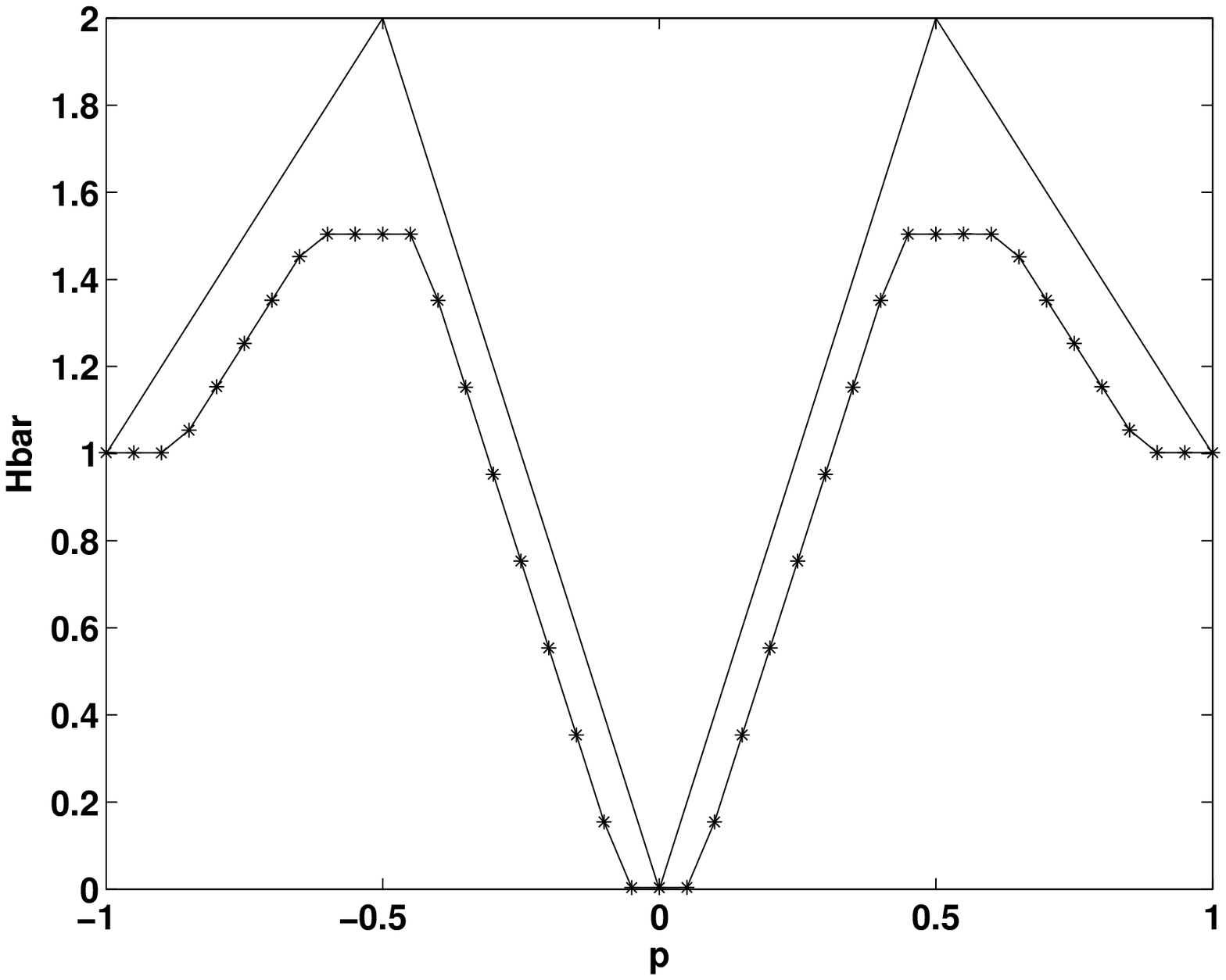}}
  (b){\includegraphics[scale=0.22]{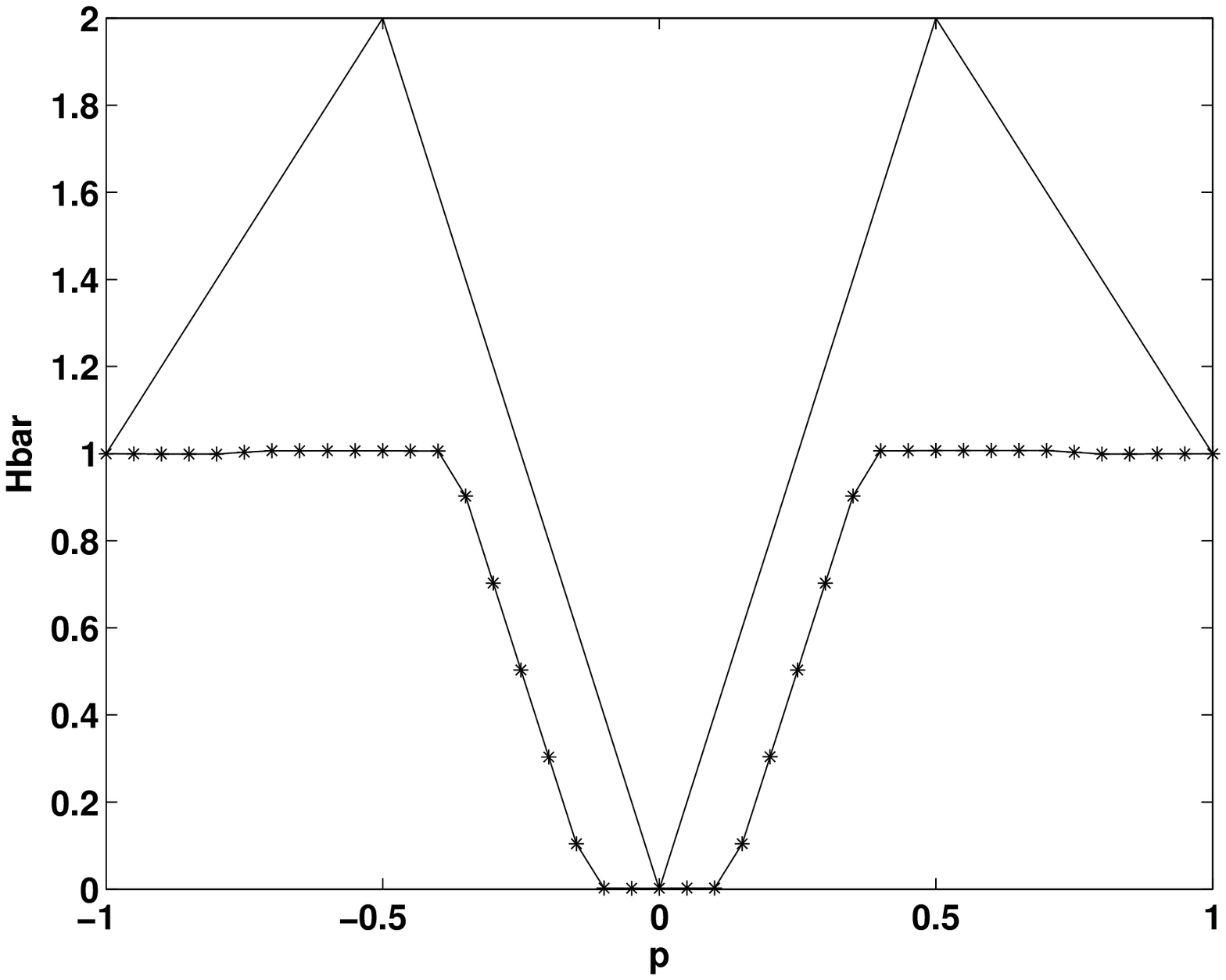}}
  (c){\includegraphics[scale=0.22]{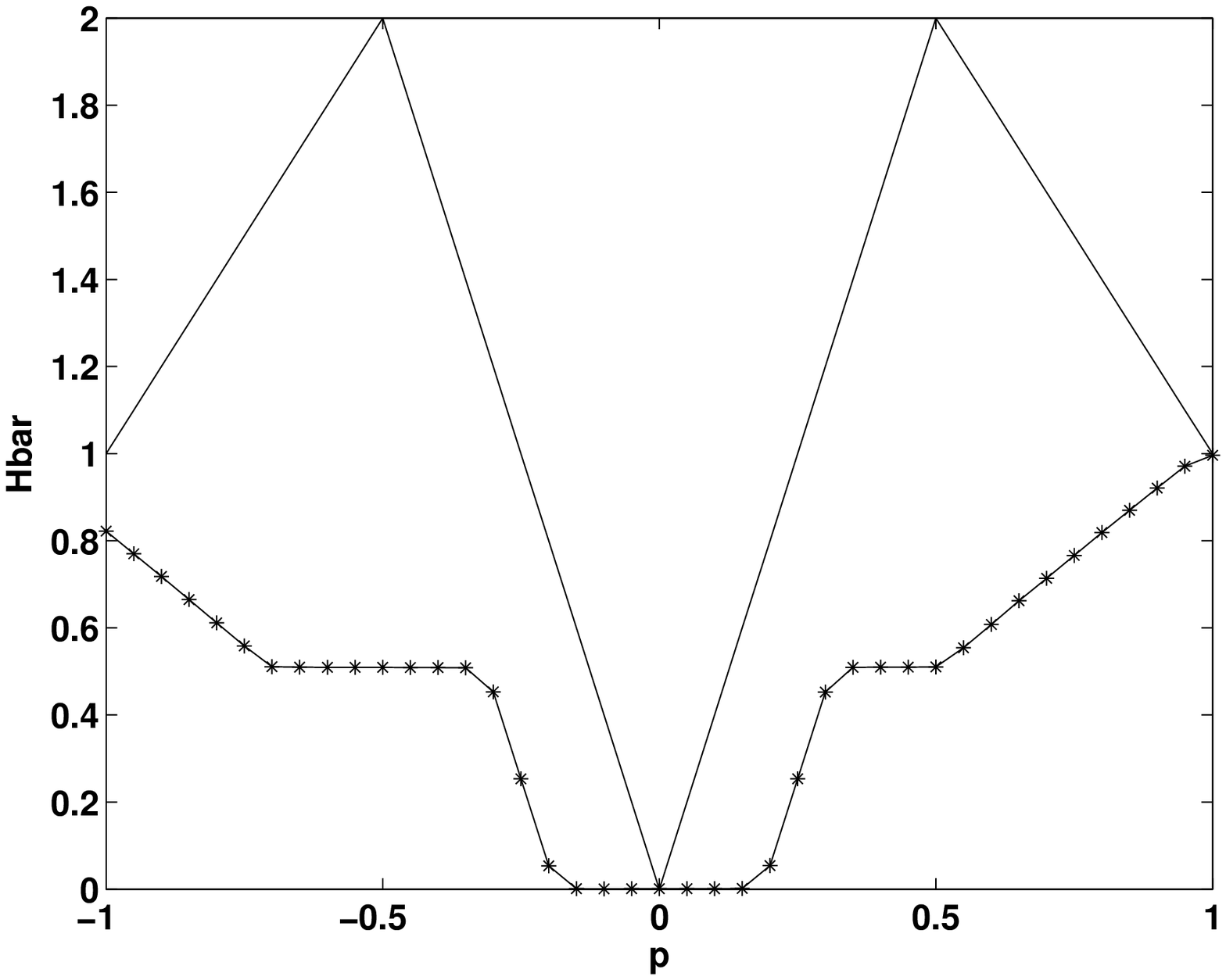}} 
  \caption{1-D case. (a): $S=0.5$. (b): $S=1.0$. (c): $S=1.50$.} 
\label{qtyEx3_2}
\end{figure}

\subsection{A double-well type Hamiltonian} \label{sec:num}  Let $n=2$. We consider a prototypical example
$$
H(p)=\min \left\{|p-e_1|, \  |p+e_1|\right\} \quad \text{ for } p=(p_1,p_2)\in \R^2,
$$
where $e_1=(1,0)$. The shape of $\ol H$ is more sensitive to the structure of the potential $V$ instead of just the oscillation. 

\subsubsection{A unstable potential} We consider the following situation
\begin{equation}\label{eq-dw}
\begin{cases}
H(p)=\min \left\{|p-e_1|, \  |p+e_1|\right\} \quad &\text{ for } p=(p_1,p_2)\in \R^2,\\
V(x)=S *(1+\sin 2\pi x_1)(1+\sin 2\pi x_2) \quad &\text{ for } x=(x_1,x_2) \in \T^2.
\end{cases}
\end{equation}
The constant $S>0$ serves as a scaling parameter to adjust the  oscillation of the potential $V$.  
Note that $V$ attains its minimum along lines $x_1=-1/4+\Z$ and $x_2=-1/4+\Z$, which is clearly not a stable situation.

For this  kind  of nonconvex Hamiltonian $H$ and $V$, complete quasi-convexification does not occur.  
However, we still see  that  $\overline H$ eventually becomes a  ``less nonconvex" function. 
Let $\overline H(p,S)$ be the effective Hamiltonian corresponding to $H(p)-V(x)$.   We have that 
\begin{thm}\label{thm:limit}
Assume that \eqref{eq-dw} holds. Then
\[
\lim_{S\to \infty}\overline H(p,S)=\max\left\{|p_2|, \  \min \{|p_1-1|, \  |p_1+1|\}\right\} \quad \text{ for } p=(p_1,p_2)\in \R^2.
\]
\end{thm}

\begin{proof}
We first show that for any $S>0$,  
\begin{equation}\label{eq:oneside}
\overline H(p,S)\geq \max\left\{|p_2|, \  \min \{|p_1-1|, \  |p_1+1|\}\right\}.
\end{equation}
Let $v$ be a viscosity solution to 
\[
H(p+Dv)-V(x)=\overline H(p,S)   \quad \text{in $\T ^2$}.
\]
Without loss of generality,  suppose that $v$ is semi-convex and  differentiable at $(x_1,  -1/4)$ for a.e.  $x_1\in  \R$.  
Otherwise,  we may use super-convolution to get a subsolution and look at a nearby line by Fubini's theorem.  Accordingly,  for a.e. $x_1\in  \R$,
\[
\overline H(p,S)  \geq \min \left\{|p_1+v_{x_1}(x_1,-1/4)-1|,  \  |p_1+v_{x_1}(x_1,-1/4)+1|\right\}.
\]
Assume that $x_1 \mapsto v(x_1,-1/4)$ attains its maximum at $x_0\in  \R$.  
Then $v_{x_1}(x_0,-1/4)=0$ due to the semi-convexity of $v$.  Hence 
\[
\overline H(p,S)\geq  \min \{|p_1-1|, \  |p_1+1|\}. 
\]
Now, similarly,  we can show that 
\[
\overline H(p,S)\geq |p_2+v_{x_2}( -1/4, x_2)|   \quad \text{for a.e. $x_2\in  \T$}.
\]
Taking the integration on both side over $[0,1]$
and using Jensen's inequality, we derive $\overline H(p,S)\geq |p_2|$.  
Thus,  (\ref{eq:oneside}) holds.

Next we show that 
\begin{equation}\label{eq:otherside}
\lim_{S\to +\infty}\overline H(p,S)\leq \max\left\{|p_2|, \  \min \{|p_1-1|, \  |p_1+1|\}\right\}.
\end{equation}
In fact, for any $\delta>0$,  it is not hard to construct $\phi\in C^1(\Bbb T^2)$ such that 
\begin{multline*}
H(p+D\phi)\leq \max\left\{|p_2|, \  \min \{|p_1-1|, \  |p_1+1|\}\right\}+\delta \\
 \text{ on $\left(\T \times \left\{-1/4\right\}\right) \bigcup \left(\left\{-1/4\right\}\times\T\right)$}.
\end{multline*}
Clearly,  
\[
\overline H(p,S)\leq \max_{x\in \T^2}  \left(H(p+D\phi(x))-S(1+\sin 2\pi x_1)(1+\sin 2\pi x_2)\right).
\]
Sending $S\to \infty$ and then $\delta\to 0$,  we obtain (\ref{eq:otherside}).

\end{proof}

\begin{rem}  Write $F_{\infty}(p)=\max\left\{|p_2|, \  \min \{|p_1-1|, \  |p_1+1|\}\right\}$.  Simple computations show that 

$\bullet$ For $r\in  [0, 1)$,   $\{F_{\infty}=r\}$ consists of two disjoint squares centered at $(\pm 1, 0)$ respectively.  

$\bullet$ For $r=1$,   $\{F_{\infty}=1\}$ consists of two adjacent squares centered at $(\pm 1, 0)$ respectively.

$\bullet$ For $r>1$,  $\{F_{\infty}=r\}$  is a rectangle centered at the origin.  

See Figure \ref{fig9} below.   
We can say that $F_{\infty}(p)$ looks more convex than the original Hamiltonian $H(p)=\min \{|p-e_1|, \  |p+e_1|\}$. 
\end{rem}

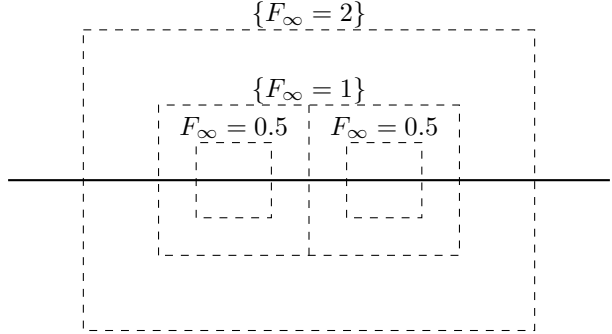
\begin{figure}[h]
\begin{center}
\begin{tikzpicture}

\draw[dashed](-1.5, -0.5)--(-1.5, 0.5)--(-0.5, 0.5)--(-0.5,-0.5)--(-1.5, -0.5);
\draw[dashed](0.5, -0.5)--(0.5, 0.5)--(1.5, 0.5)--(1.5,-0.5)--(0.5, -0.5);
\draw[dashed](-2,-1)--(-2,1)--(2,1)--(2,-1)--(-2,-1);
\draw[dashed](0,1)--(0,-1);
\draw[dashed](-3,2)--(3,2)--(3,-2)--(-3,-2)--(-3,2);
\draw (0,2.2) node {\footnotesize $\{F_{\infty}=2\}$};
\draw (0, 1.2) node {\footnotesize $\{F_{\infty}=1\}$};
\draw (1,0.7)  node {\footnotesize $F_{\infty}=0.5$};
\draw (-1,0.7)  node {\footnotesize $F_{\infty}=0.5$};
\draw[thick](-4,0)--(4,0);

\end{tikzpicture}
\caption{Level curves of $F_{\infty}$}
\label{fig9}
\end{center}

\end{figure}

\smallskip

\subsubsection{A stable potential} 
We consider the following 
\begin{equation}\label{eq-dw2}
\begin{cases}
H(p)=\min \left\{|p-e_1|, \  |p+e_1|\right\} \quad &\text{ for } p=(p_1,p_2)\in \R^2,\\
V(x)=S * \left(\sin^2 (2\pi x_1)+\sin^2 (2\pi x_2)\right) \quad &\text{ for } x=(x_1,x_2) \in \T^2.
\end{cases}
\end{equation}
The constant $S>0$ serves as a scaling parameter to adjust the  oscillation of the potential $V$.  
 Note that $V$ attains its minimum at points $(x_1,x_2)$ such that $2(x_1,x_2)\in \Z^2$, which is  a stable situation.

For this kind of potential,  it is easy to show that 
$$
\lim_{S\to \infty}\ol H(p,S)=0  \quad \text{locally uniformly in $\R^2$}.
$$
More interestingly,  numerical computations (Figures \ref{qtyEx1_2_weno3a} and \ref{qtyEx1_2_weno3b}) 
below suggest that  $\ol H$ becomes quasiconvex at least when $S\geq 1$. 

\smallskip

\begin{quest} \label{quest2}
Assume that \eqref{eq-dw2} holds.
Does there exist $L$ such that when $S>L$,  $\ol H$ is quasiconvex?  

\end{quest}

\smallskip

\noindent{\bf Numerical example 3.}  We consider setting \eqref{eq-dw2}.
 The constant $S$ serves as the scaling parameter to increase or decrease the effect of the potential. 

We use the Lax-Friedrichs based big-T method to compute the effective Hamiltonian. 
The computation for time-dependent Hamilton-Jacobi equations is done by using the LxF-WENO 3rd-order scheme. The initial condition for the big-T method is taken to be $\cos(2\pi x_1)\sin(2\pi x_2)$.  See Figure \ref{qtyEx1_2_weno3a}  and Figure \ref{qtyEx1_2_weno3b} for results. 

\begin{figure}[htb]
 \centering
 (a){\includegraphics[scale=0.22]{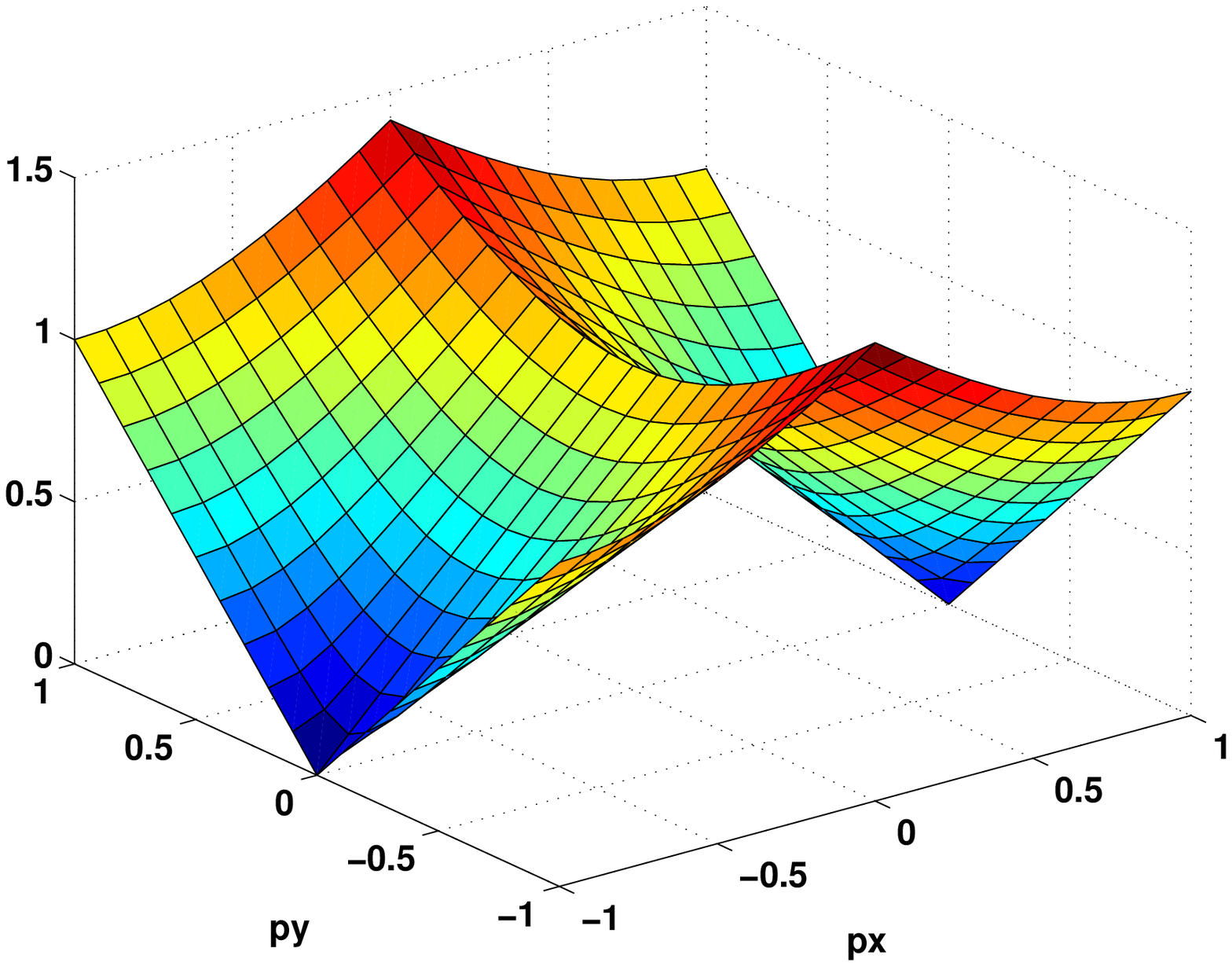}}
 (b){\includegraphics[scale=0.22]{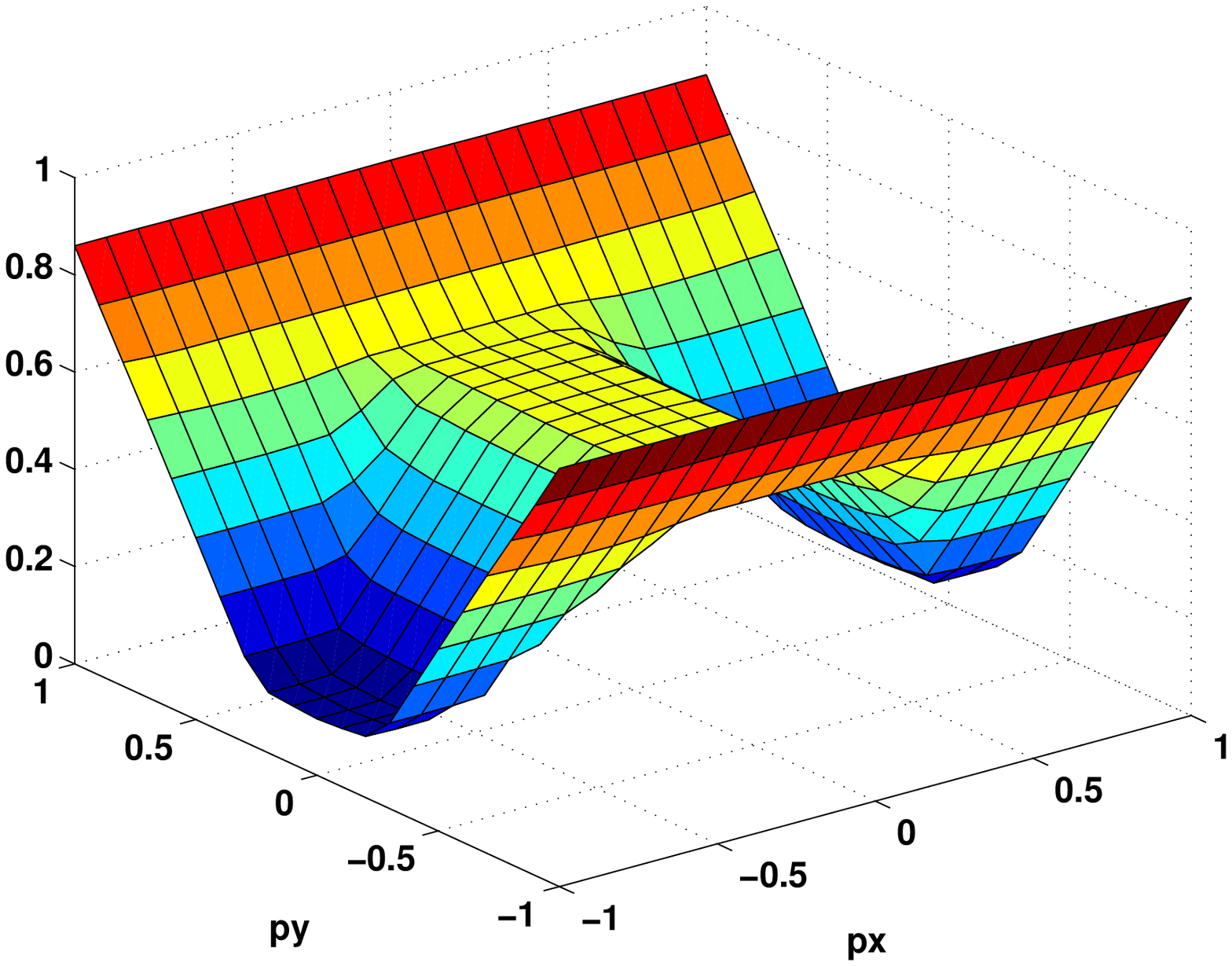}}
 (c){\includegraphics[scale=0.22]{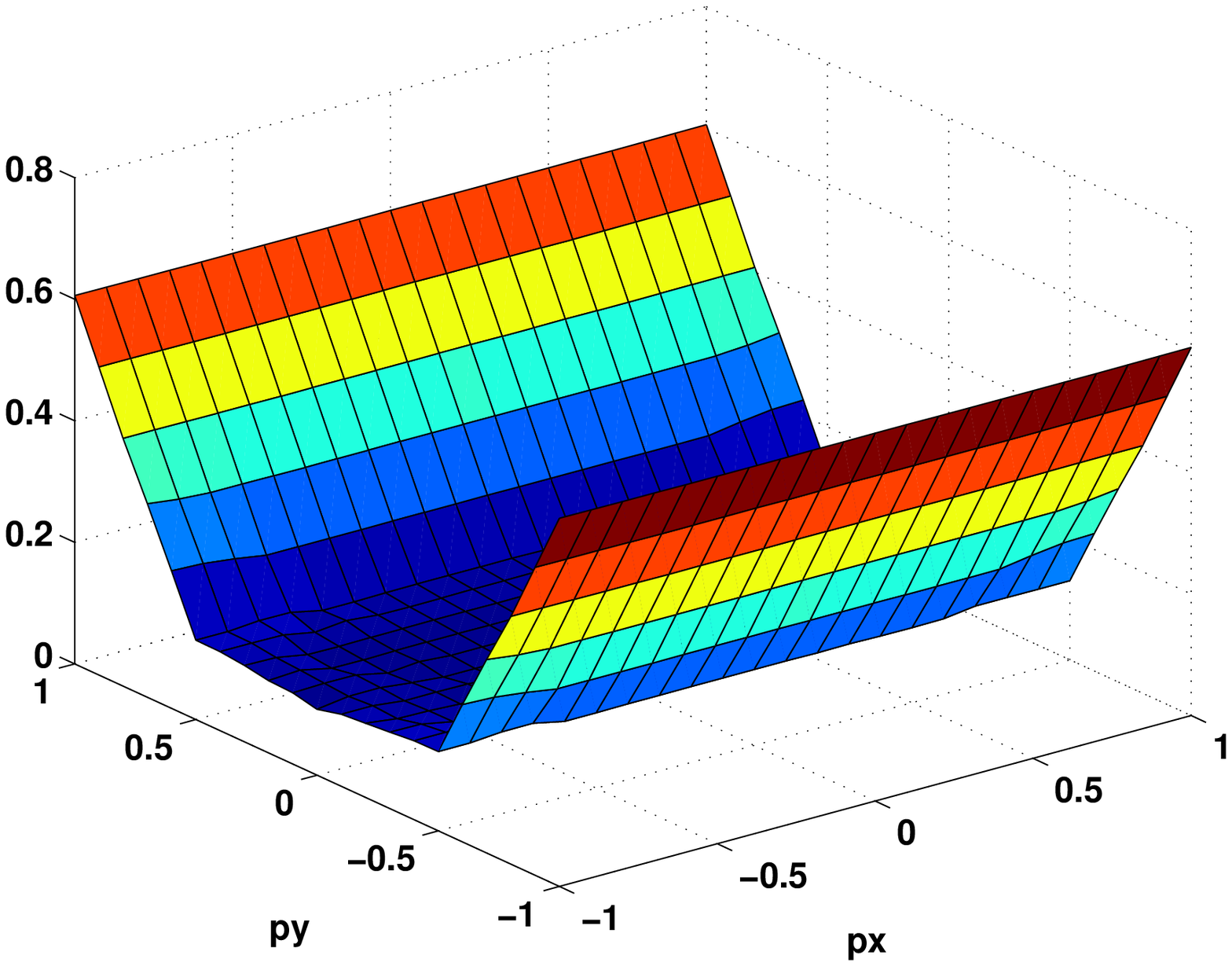}}
 (d){\includegraphics[scale=0.22]{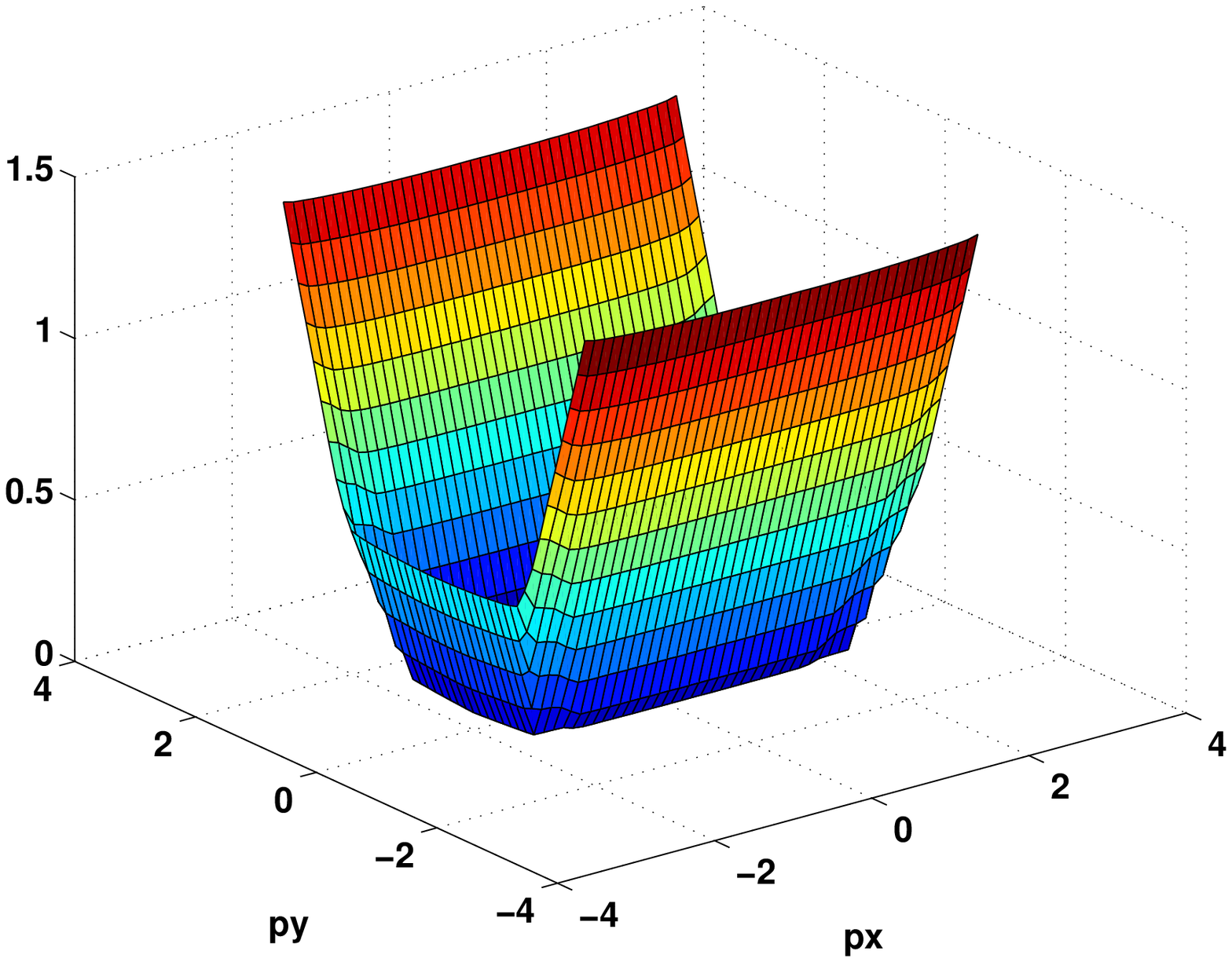}}
 (e){\includegraphics[scale=0.22]{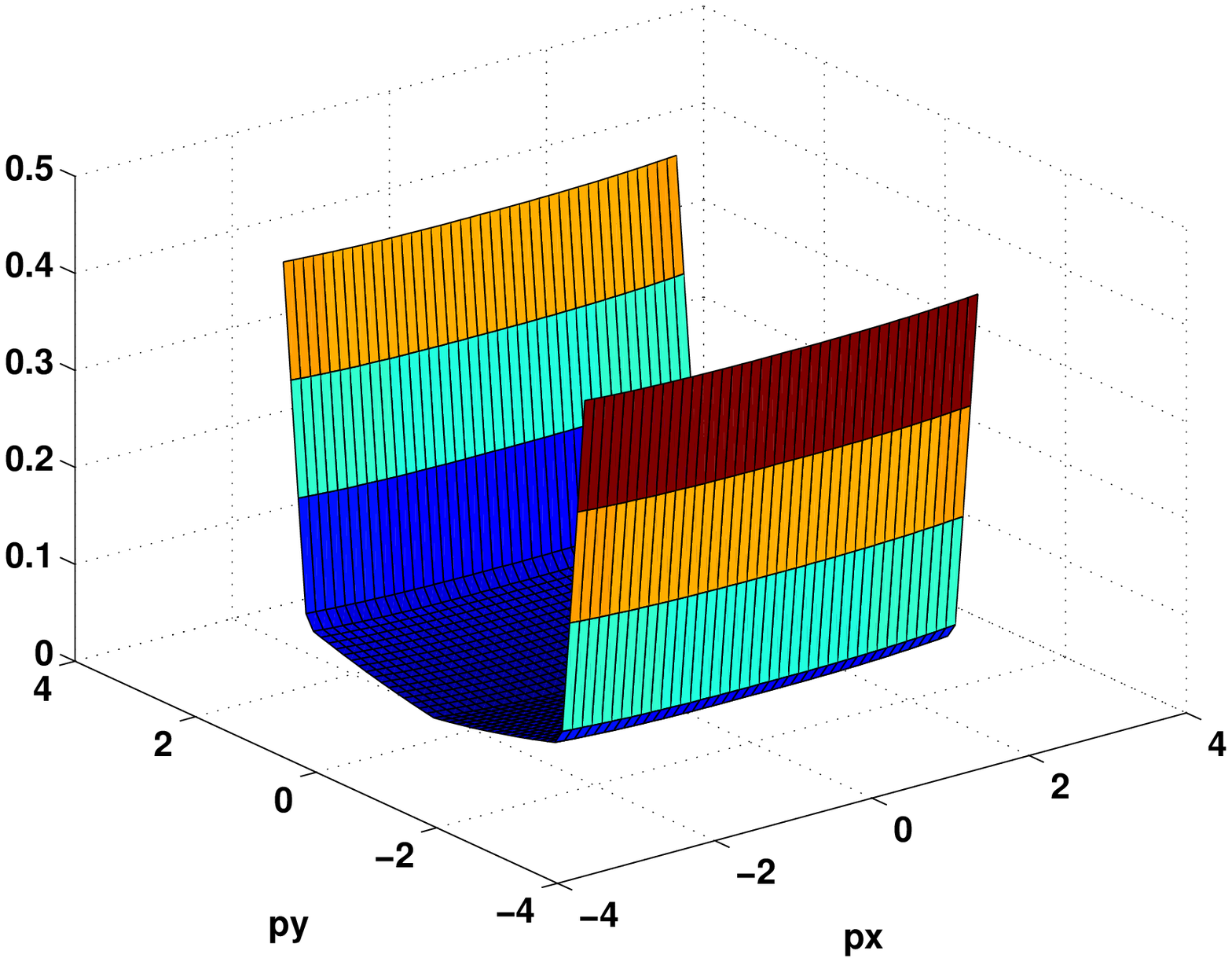}}
  \caption{Surface plots. (a) The original Hamiltonian: $S=0$. (b) $S=0.50$. (c)  $S=1.0$.  (d) $S=2.0$. (e)  $S=4.0$.}
\label{qtyEx1_2_weno3a}
\end{figure}

\begin{figure}[htb]
 \centering
 (a){\includegraphics[scale=0.22]{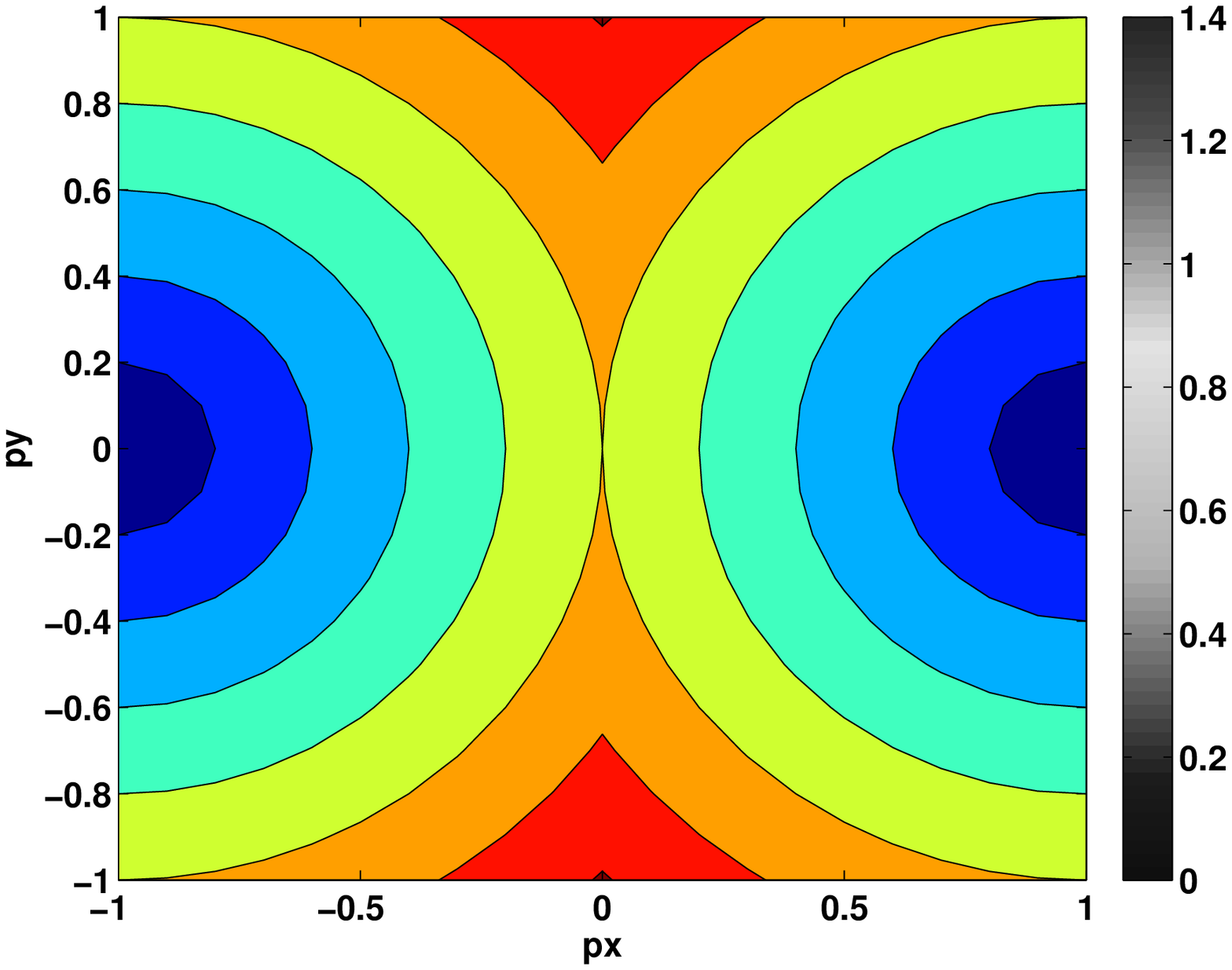}}
 (b){\includegraphics[scale=0.22]{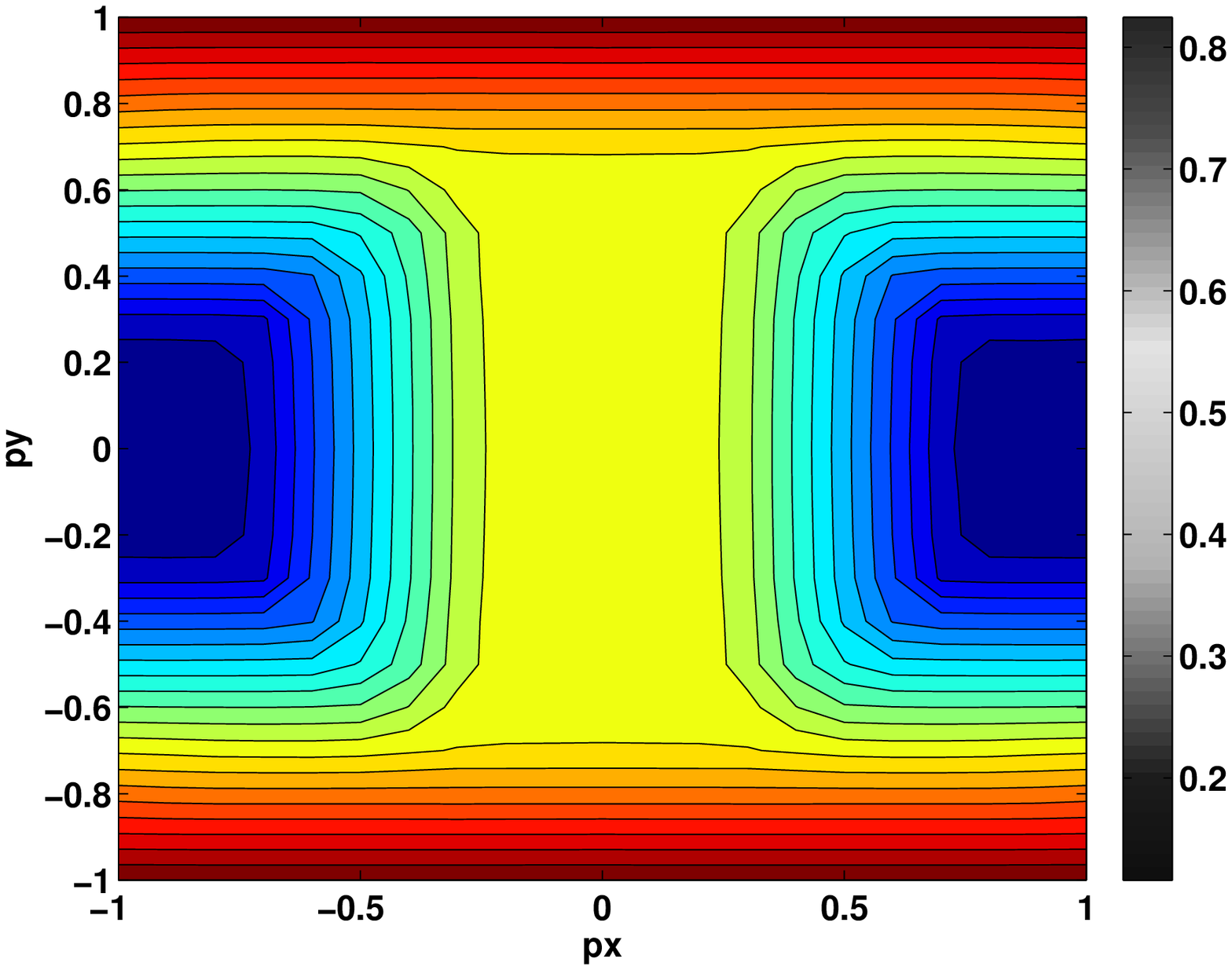}}
 (c){\includegraphics[scale=0.22]{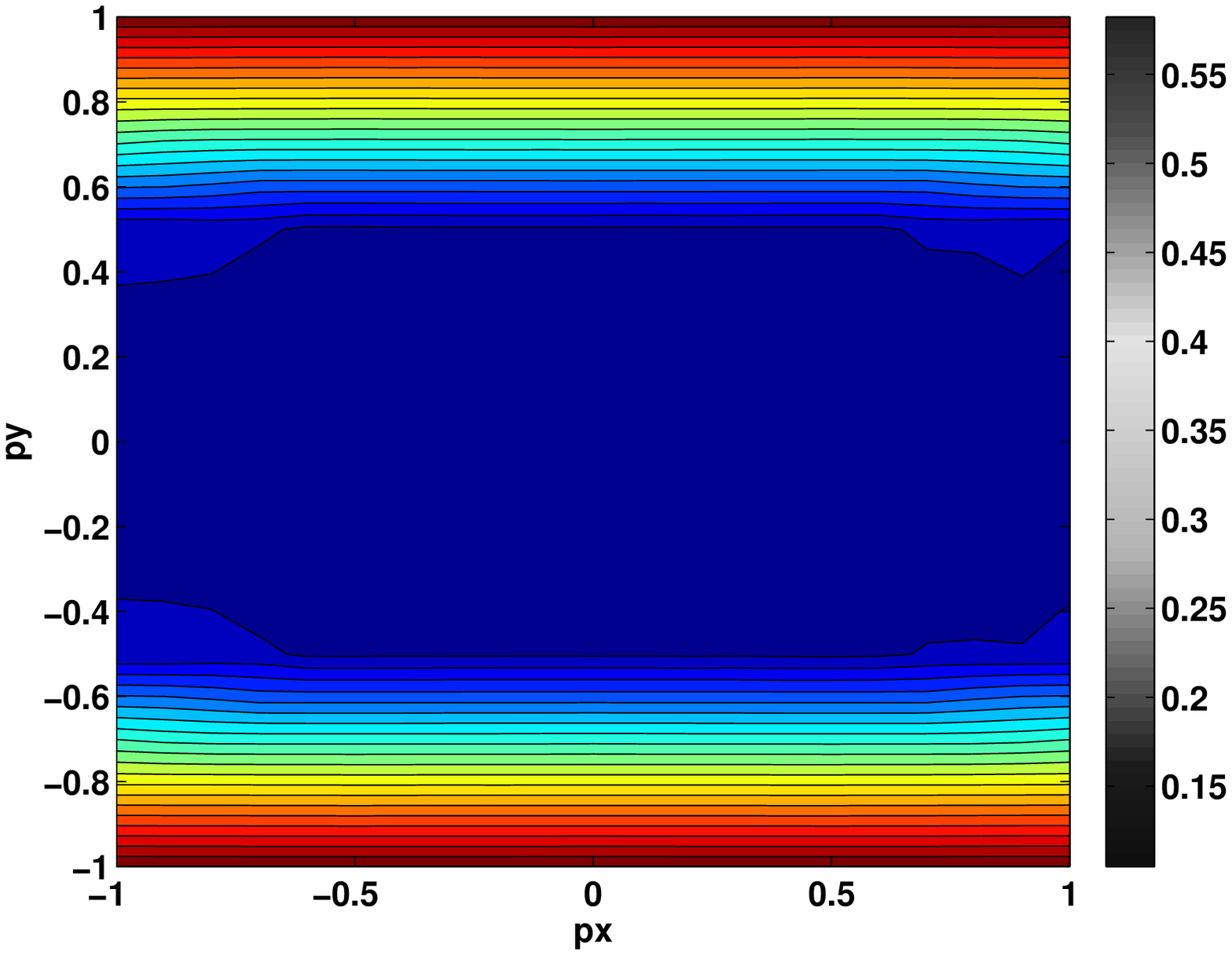}}
 (d){\includegraphics[scale=0.22]{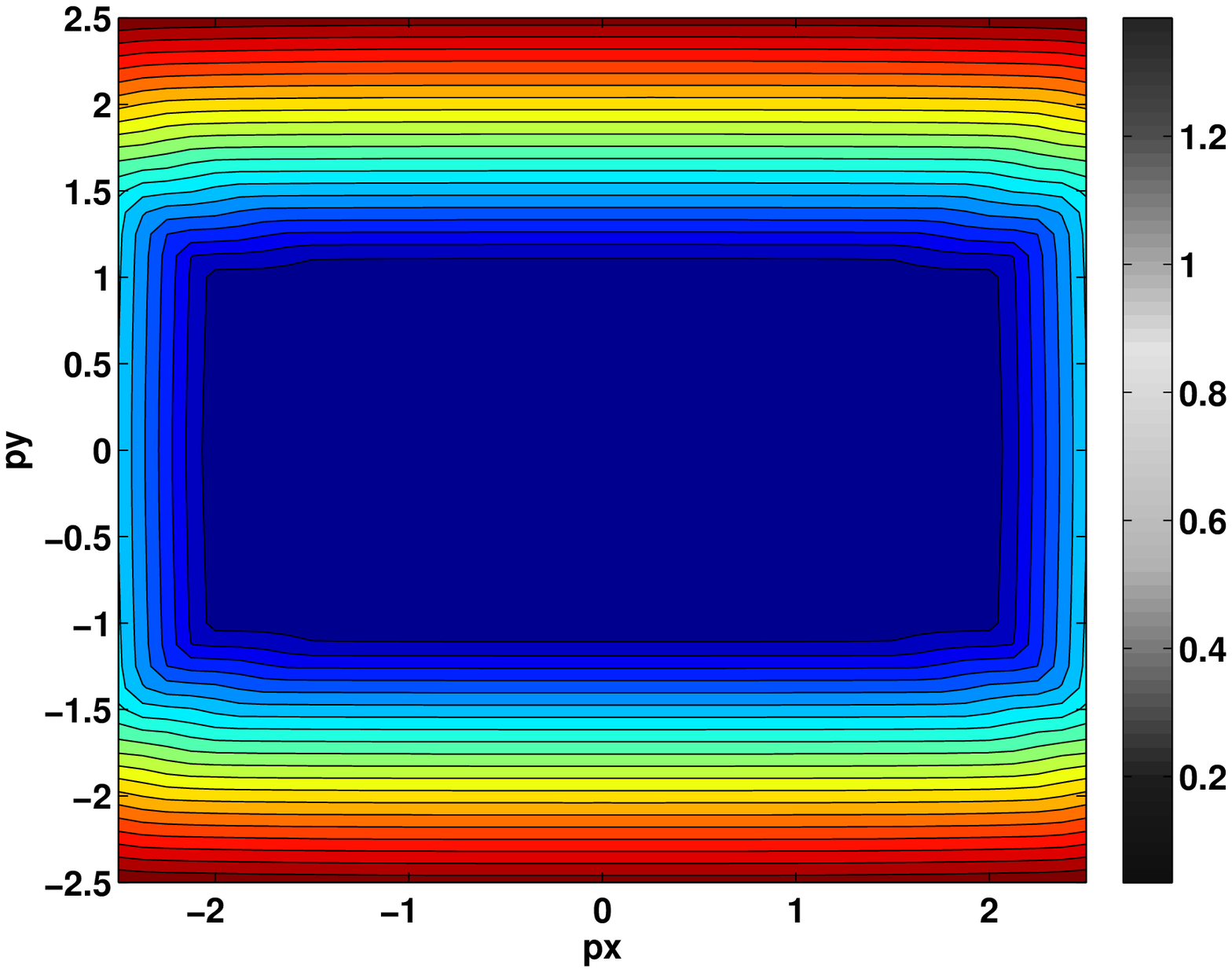}}
 (e){\includegraphics[scale=0.22]{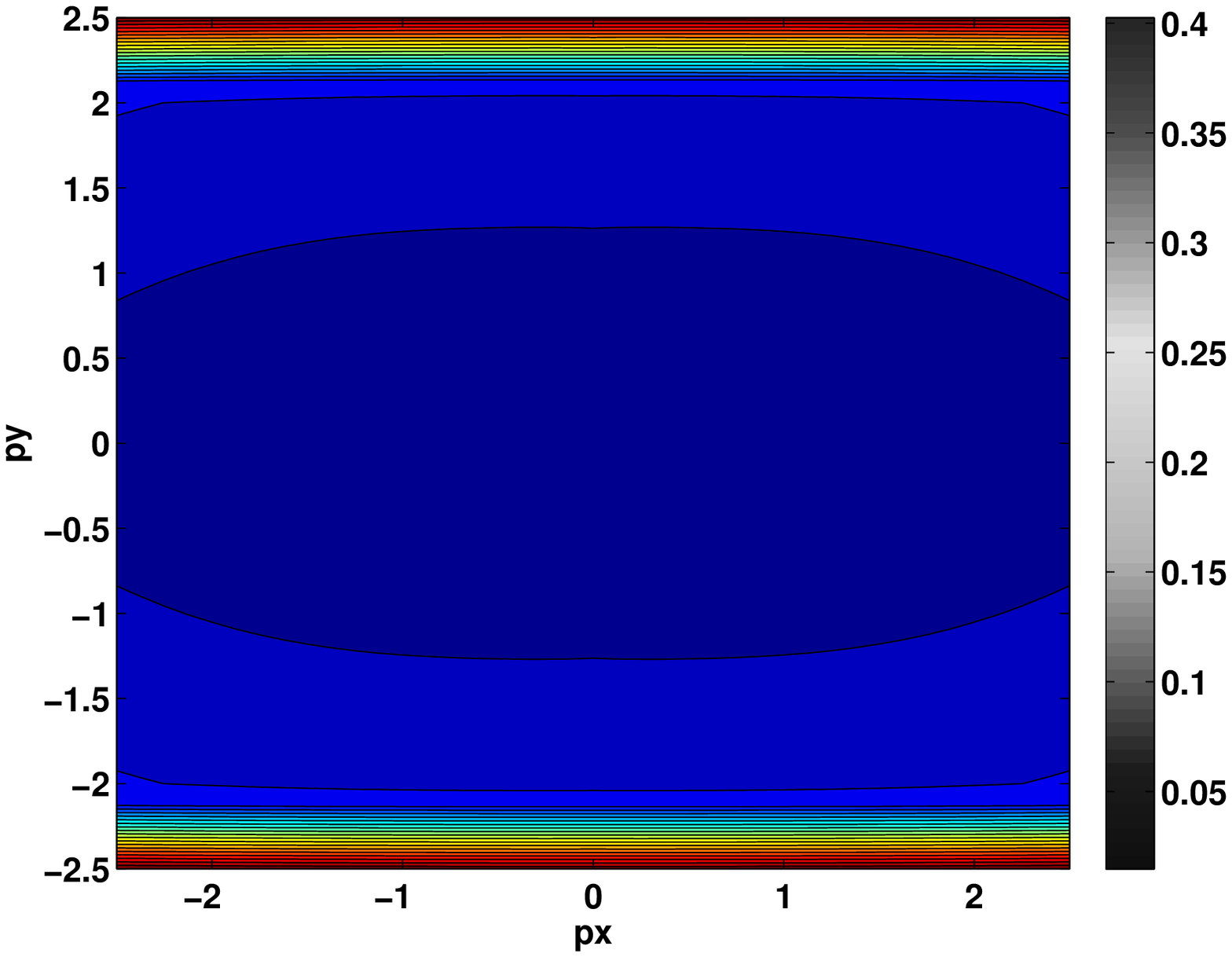}} \caption{Contour plots. (a) The original Hamiltonian: $S=0$. (b) $S=0.50$. (c)  $S=1.0$. (d) $S=2.0$. (e)  $S=4.0$.}
\label{qtyEx1_2_weno3b}
\end{figure}

\medskip

%

\noindent{\bf Numerical example 4.}  We consider 
\begin{equation*}\label{eq-dw3}
\begin{cases}
H(p)=\min \left\{|p-e_1|, \  |p+e_1|\right\} \quad &\text{ for } p=(p_1,p_2)\in \R^2,\\
V(x)=S * \left(3+\sin (2\pi x_1)+\sin (4\pi x_1)+\sin (2\pi x_2)\right) \quad &\text{ for } x=(x_1,x_2) \in \T^2.
\end{cases}
\end{equation*}
This is the case that $V$ is not even. The constant $S$ serves as the scaling parameter to increase or decrease the effect of the potential. See Figure \ref{qtyEx1_5_weno3d} below. Clearly, $\ol H$ is not even when $S=0.125$,   $S=0.25$, $S=0.5$, $S=0.95$ and $S=1.0$. Loss of evenness for all $S$ implies that $\ol H$ can not  have a decomposition formula like (\ref{uni-decom}) regardless of the oscillation of the $V$.  
 
\begin{figure}[htb]
 \centering
 (a){\includegraphics[scale=0.22]{OriHamLevelsetKeyHam1Mp21.eps}}
 (b){\includegraphics[scale=0.22]{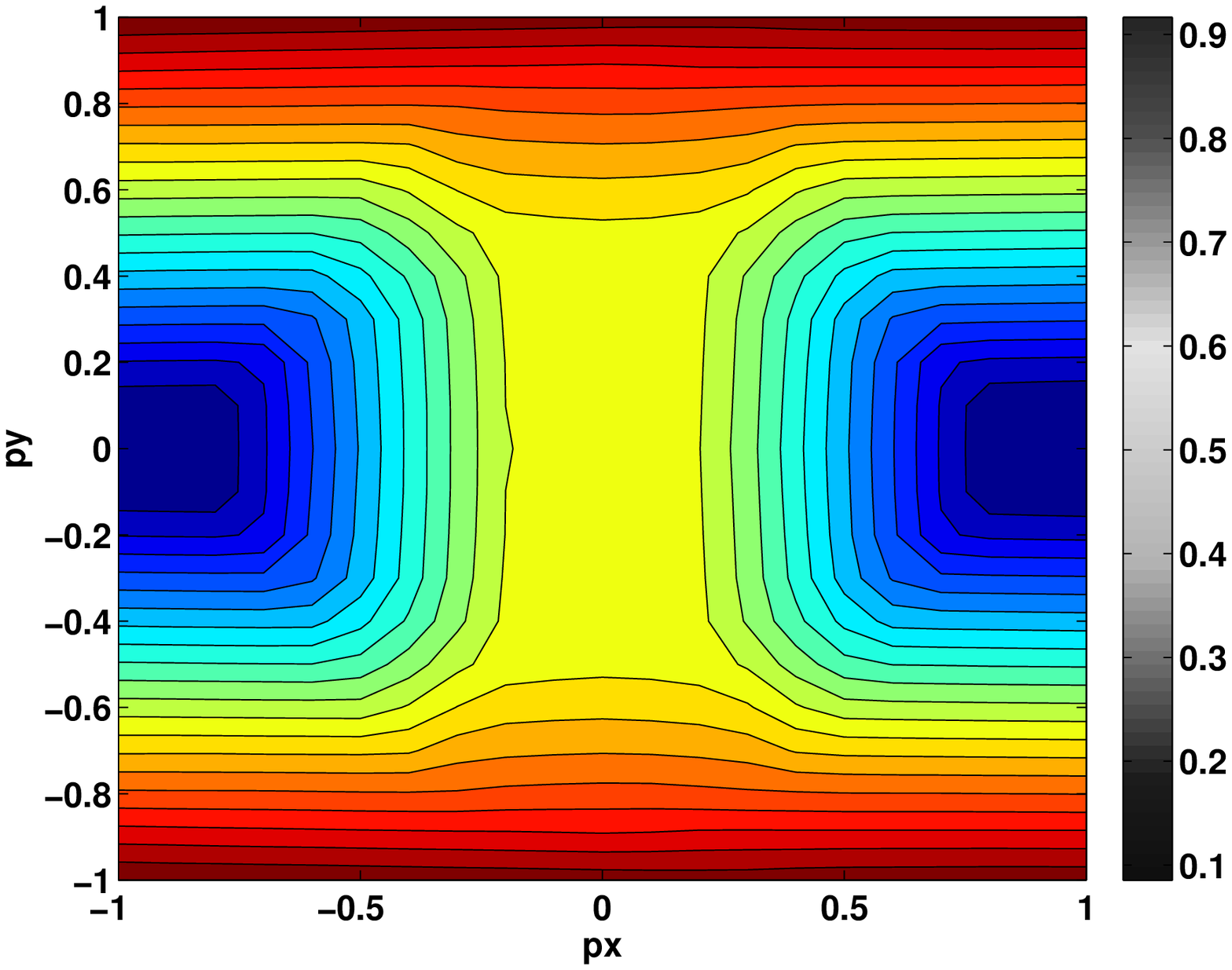}}
 (c){\includegraphics[scale=0.22]{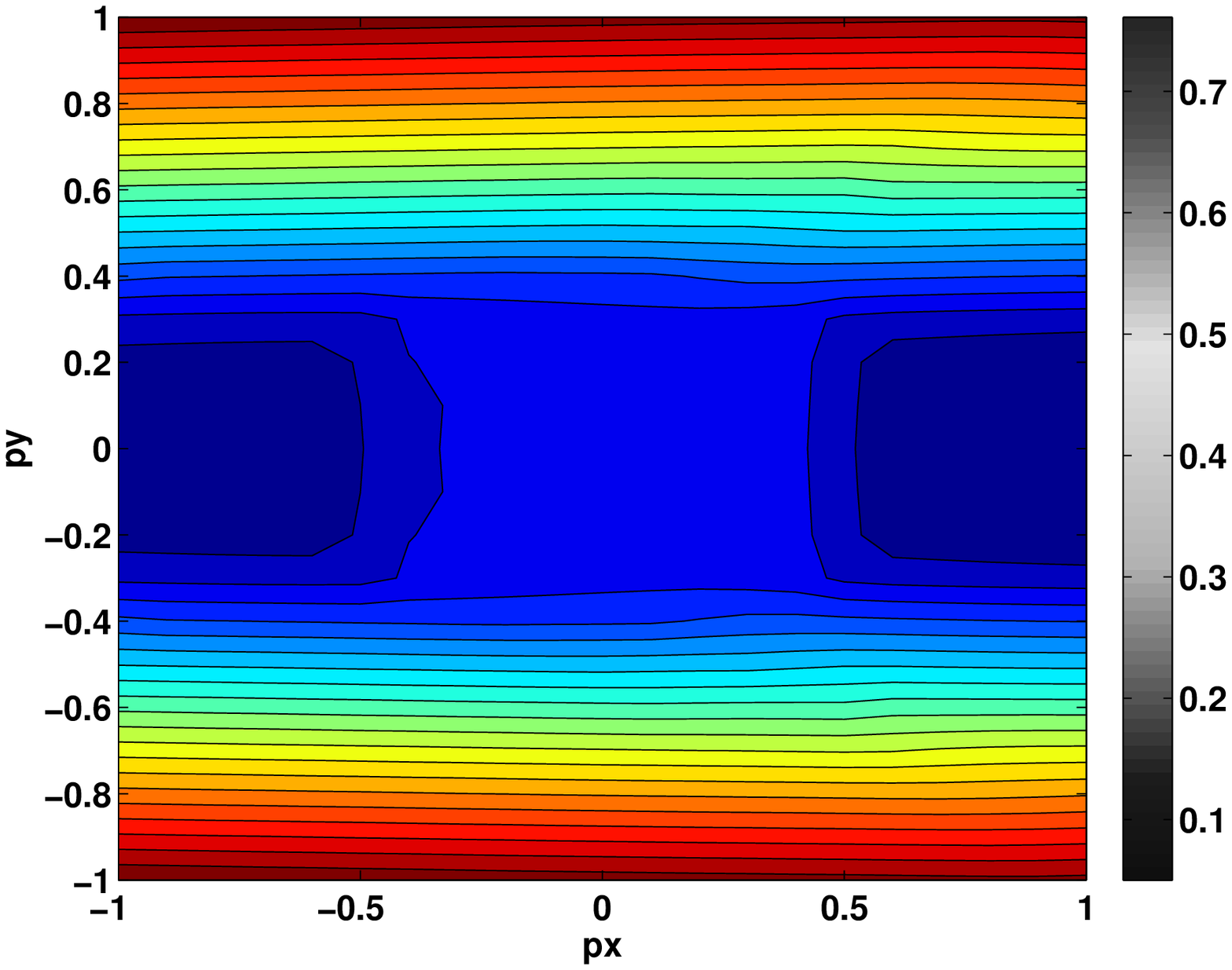}}\\
 (d){\includegraphics[scale=0.22]{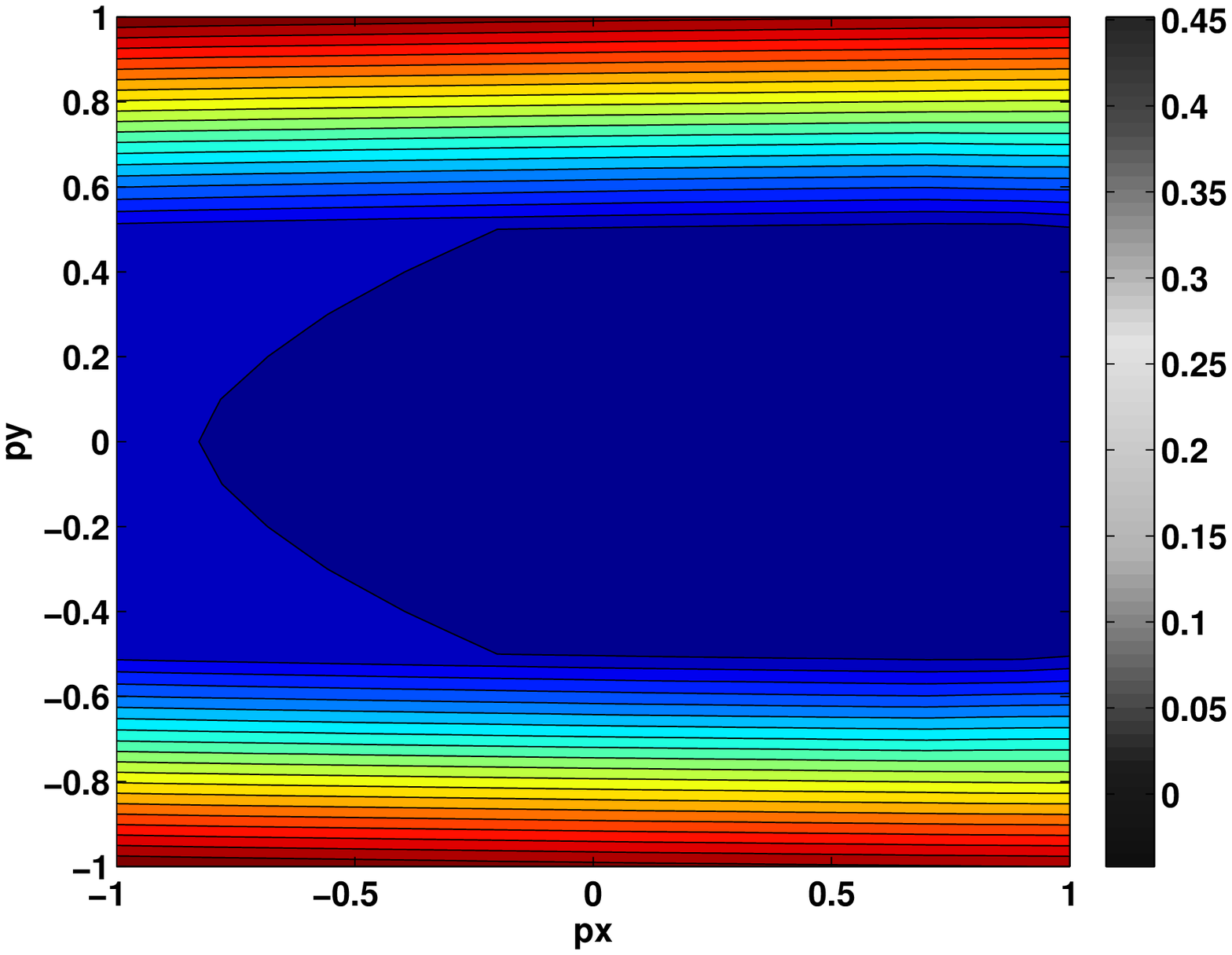}} 
 (e){\includegraphics[scale=0.22]{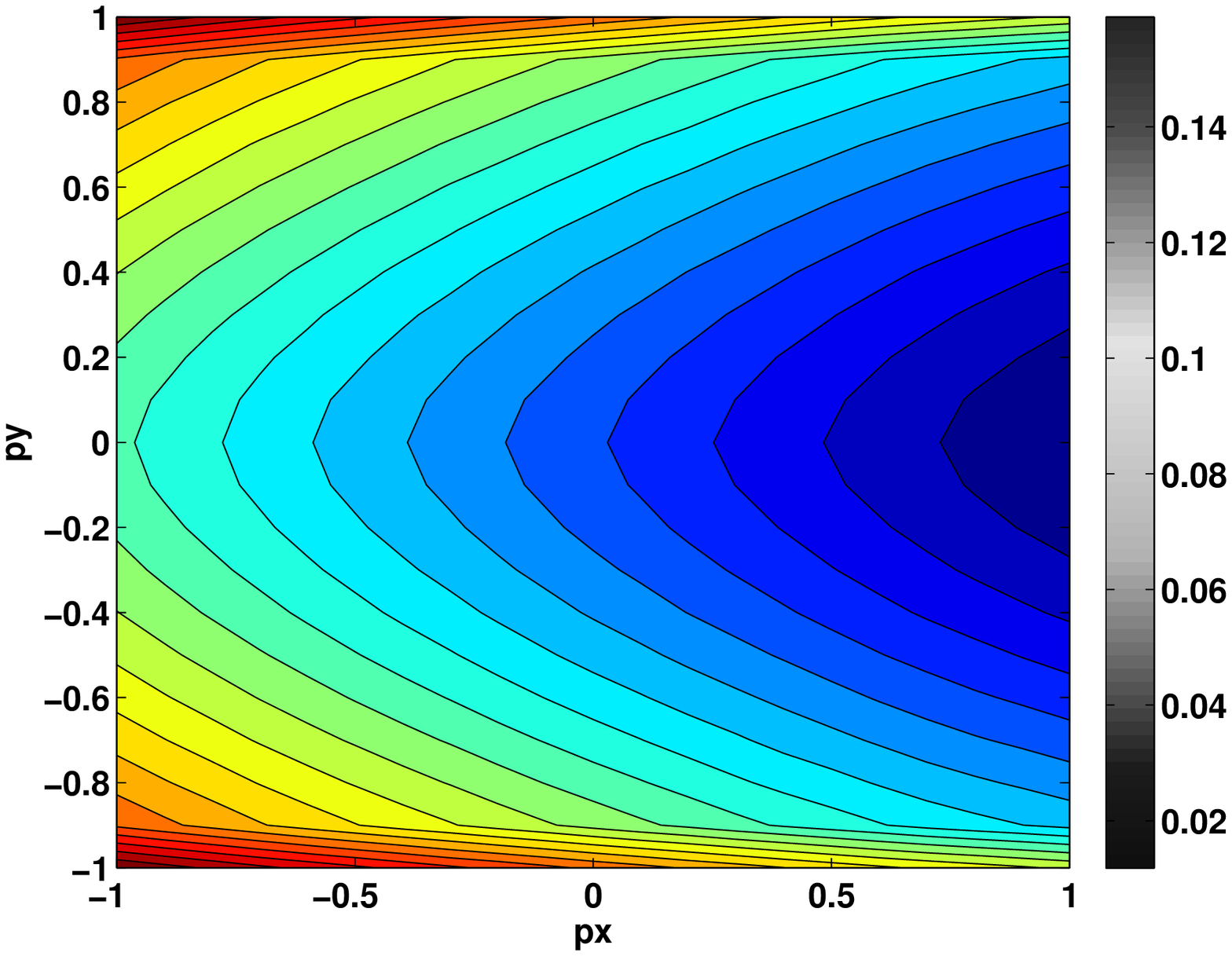}}
 (f){\includegraphics[scale=0.22]{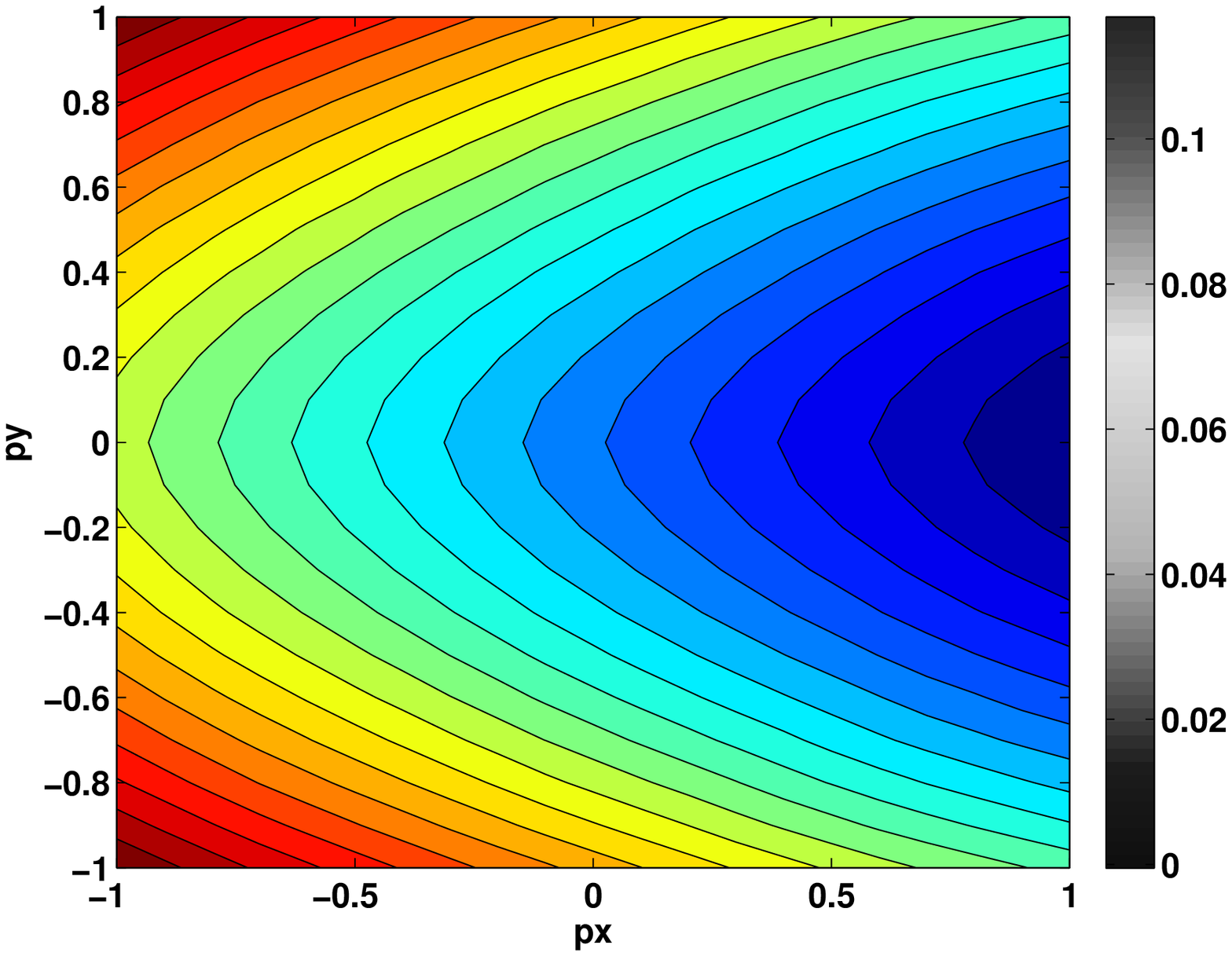}}\\  
\caption{(a) The original Hamiltonian: $S=0$. (b) $S=0.125$. (c) $S=0.25$. (d) $S=0.5$. (e) $S=0.95$. (f) $S=1.0$.}
\label{qtyEx1_5_weno3d}
\end{figure}

\section{Appendix: Some application in Random homogenization}\label{sec:Appen}
As a bypass product, we show that all Hamiltonians in Section \ref{sec:min-max}
are actually regularly homogenizable in the stationary ergodic setting.
Let us first give a brief overview of stochastic homogenization.

\subsection{Brief overview of stochastic homogenization}
 Let $(\Om, \cF, \bP)$ be a probability space.   
 Suppose that $\{\tau_y\}_{y\in \R^n}$ is a  measure-preserving translation group action of $\R^n$ on $\Om$ which satisfies that 
\begin{itemize}
\item[(1)] ({\bf Semi-group property})
$$
\tau_x\circ \tau_y=\tau_{x+y} \quad \text{ for all } x,y\in \R^n.
$$
\item[(2)] ({\bf Ergodicity}) For any $E\in \cF$, 
$$
\tau_x(E)=E \quad   \text{ for all } x\in \R^n  \quad \Rightarrow  \quad \bP(E)=0 \quad \mathrm{or} \quad \bP(E)=1.
$$
\end{itemize}
The potential $V(x, \omega):\R^n\times \Om \to \R$ is assumed to be stationary,  bounded and uniformly continuous.
More precisely, $V(x+y,\omega)=V(x, \tau_y \omega)$ for all $x,y \in \R^n$ and $\om \in \Om$,  
$\text{ess} \sup_{\Om}|V(0,\omega)|<+\infty$ and
$$
|V(x,\omega)-V(y,\omega)|\leq c(|x-y|) \quad \text{ for all $x,y \in \R^n$ and $\om \in \Om$},
$$
for  some function $c:[0, \infty)\to  [0, \infty)$ satisfying $\lim_{r\to 0}c(r)=0$. 

For $\ep>0$, denote $u^{\ep}(x,t, \omega)$ as the unique viscosity solution to 
\begin{equation}\label{RamHJ-eq}
\begin{cases}
u^\ep_t + H(Du^\ep) - V\left(\frac{x}{\ep}, \omega\right)=0 \quad &\text{ in } \R^n \times (0,\infty),\\
u^\ep(x,0, \omega)=g(x) \quad &\text{ on } \R^n.
\end{cases}
\end{equation}
Here $H\in C(\R^n)$ is coercive.  A basic question is whether $u^{\ep}$, as $\ep \to 0$,  converges to the solution to an effective deterministic equation (\ref{HJ-hom}) almost surely as in the periodic setting.

The stochastic homogenization of Hamilton-Jacobi equations   has received much attention in the last seventeen years. 
The first  results were due to Rezakhanlou and Tarver \cite{ReTa} and Souganidis \cite{Sou}, 
who independently proved convergence results for general convex, first-order Hamilton-Jacobi equations in stationary ergodic setting. 
These results were extended to the viscous case with convex Hamiltonians by Kosygina, Rezakhanlou and Varadhan \cite{KRV} and, independently, 
by Lions and Souganidis \cite{LiS1}. 
New proofs of these results based on the notion of intrinsic distance functions (maximal subsolutions) 
appeared later in Armstrong and Souganidis \cite{AS3} for the first-order case 
and in Armstrong and Tran \cite{AT1} for the viscous case.
See Davini, Siconolfi \cite{DaSi}, Armstrong, Souganidis \cite{AS3} for homogenization of quasiconvex, first-order Hamilton-Jacobi equations.
\smallskip

One of the prominent  open problems in the field is to prove/disprove homogenization in the genuinely nonconvex setting. 
In \cite{ATY1}, Armstrong, Tran and Yu showed that, for $H(p)=(|p|^2-1)^2$, \eqref{RamHJ-eq}
homogenizes  in all space dimensions $n\geq 1$. 
In the next paper \cite{ATY2},
Armstrong, Tran and Yu proved that, for $n=1$,  \eqref{RamHJ-eq} homogenizes for general coercive $H$. 
Gao \cite{Gao} generalized the result in \cite{ATY2} to the general non separable Hamiltonians $H(p,x,\om)$ in one space dimension. 
{\it A common strategy  in  papers \cite{ATY1, ATY2,  Gao} is to identify the shape of $\ol H$  in the periodic setting first and then recover it in the stationary ergodic setting. } In particular,  in contrast to previous works,  our strategy does not depend on finding some master ergodic quantities suitable for subadditive ergodic theorems.  Such kind of ergodic quantities may not exist at all for  genuinely nonconvex $H$.  

\smallskip

Ziliotto \cite{Zi} gave a counterexample to homogenization of \eqref{RamHJ-eq} in case $n=2$.
See also the paper by Feldman and Souganidis \cite{Fe-Sou}.
Basically, \cite{Zi,Fe-Sou} show that, if $H$ has a strict saddle point, then there exists a potential energy $V$ such that
$H-V$ is not homogenizable.
\smallskip

Based on min-max formulas established in Section \ref{sec:min-max},  we prove that, for the Hamiltonians $H$ appear in Theorem \ref{thm:rep1}, Corollary \ref{cor:rep3},
 Lemma \ref{thm:rep2}, and Theorem \ref{Maintheorem}, $H-V$ is always regularly homogenizable in all space dimensions $n \geq 1$. See the precise statements in Theorems \ref{thm:random}, \ref{thm:random-m}, and Corollary \ref{cor:random} in Subsection \ref{sec:stoc-main}.
Theorem \ref{thm:random} includes the result in \cite{ATY1} as a special case.  Also, 
the result of Corollary \ref{cor:random}  implies that, in some specific cases,
even if $H$ has strict saddle points, $H-V$ is still regularly homogenizable  for every $V$ with large oscillation. See the comments after its statement and some comparison between this result and the counterexamples in \cite{Zi, Fe-Sou}.  

\smallskip

The authors tend to believe that a prior identification of the shape of $\ol H$ might be necessary in order to tackle homogenization in the general stationary ergodic setting.   In certain special random environment like  finite range dependence (i.i.d),  the homogenization was established for a class of Hamiltonians in  interesting works of Armstrong, Cardaliaguet \cite{AC}, Feldman, Souganidis \cite{Fe-Sou}.  Their proofs are based on completely different philosophy and, in particular,  rely on specific properties of the random media.  

\smallskip

In the viscous case (i.e., adding $-\ep \Delta u^{\ep}$ to equation (\ref{RamHJ-eq})), the stochastic homogenization problem for nonconvex Hamiltonians is more formidable.  For example, the homogenization has not even been proved or disproved for  simple cases like $H(p,x)=(|p|^2-1)^2-V(x)$ in one dimension.  Min-max formulas in the inviscid case are in general not available here due to the nonlocal effect (or regularity)  from the viscous term.  Nevertheless, see  a preliminary result in one dimensional space by Davini and Kosygina \cite{DaKo}.

\smallskip

The following definition was first introduced in  \cite{ATY2}.
\begin{defn}\label{regularhomo}  
We say that $H-V$  is regularly homogenizable if  for every  $p\in  \R^n$,  there exists a unique constant $\overline H(p)\in   \R$ such that,
for every $R>0$ and for a.s. $\omega\in \Omega$,
\begin{equation}\label{ergo0}
\limsup_{\lambda\to 0}\max_{|x|\leq {R\over \lambda}}\left|\lambda v_{\lambda}(x, p, \omega)+\overline H(p)\right|=0.
\end{equation}
Here, for $\lam>0$,  $ v_{\lambda}(\cdot, p, \omega)\in W^{1,\infty}(\R^n)$ is the unique bounded viscosity solution to
$$
\lambda v_{\lambda}+H(p+Dv_{\lambda})-V(x, \omega)=0   \quad \text{in $\R^n$}.
$$
\end{defn}

According to Lemma 5.1 in  \cite{AS3},  \eqref{ergo0}  is equivalent to saying that for a.s. $\omega\in \Omega$,
$$
\lim_{\lambda\to 0}\left|\lambda v_{\lambda}(0, p, \omega)+\overline H(p)\right|=0.
$$
Clearly,  if $H-V$ is regularly homogenizable,  then random homogenization holds,
that is,  solution $u^{\ep}$ of \eqref{RamHJ-eq} converges to $u$, the solution to \eqref{HJ-hom} with $\ol{H}$ defined by \eqref{ergo0}, as $\ep \to 0$.

\subsection{Stochastic homogenization results}\label{sec:stoc-main}
The main claim is that $H(p)-V(x,\om)$ is regularly homogenizable provided that $H(p)$ is of a form in Theorem \ref{thm:rep1}, Corollary \ref{cor:rep3},
 Lemma \ref{thm:rep2}, and Theorem \ref{Maintheorem}. The proof is basically a repetition of arguments in the proofs of the aforementioned results 
except that the cell problem in the periodic setting is replaced by the discount ergodic problem in the random environment.   
This is because of the fact that the cell problem in the random environment might not have sublinear solutions at all.
Here are the precise statements of the results.

\begin{thm} \label{thm:random}
Assume that $H\in C(\R^n)$ satisfies {\rm (H1)--(H3)}.
Assume further that $ {\rm ess}\inf_{\Omega}V(0,\omega)=0$.
Then $H-V$ is regularly homogenizable and
\[
\ol{H} =\max\left\{ \ol{H}_1, \ol{H}_2,0\right\}.
\]
\end{thm}
As mentioned, this theorem includes the result in \cite{ATY1} as a special case.
An important corollary of this theorem is the following:
\begin{cor}\label{cor:random}   
Let $H\in C(\R^n)$ be a coercive Hamiltonian satisfying {\rm (H1)--(H3)}, 
except that we do not require $H_2$ to be quasiconcave.   
Assume  that  $ {\rm ess}\inf_{\Omega}V(0,\omega)=0$, $ {\rm ess}\sup_{\Omega}V(0,\omega)=\ol{m}$, and 
\[
\ol{m}>  \max_{\ol{U}}H=\max_{\R ^n}H_2.
\] 
Then $H-V$ is regularly homogenizable and
\[
\ol H=\max\left\{ \ol H_1,   0 \right\}.
\]
In particular,  $\ol H$ is quasiconvex in this situation. 
\end{cor}
It is worth emphasizing that
we do not require any structure of $H$ in $U$ except that $H>0$ there.
In particular, $H$ is allowed to have strict saddle points in $U$.
Therefore, Corollary \ref{cor:random} implies that, in some specific cases,
even if $H$ has strict saddle points, $H-V$ is still regularly homogenizable provided that 
the oscillation of $V$ is large enough.
In  a way, this is a situation when the potential energy $V$ has much power to overcome 
the depths of all the wells created by the kinetic energy $H$ and it ``irons out" all the nonconvex pieces
to make $\ol{H}$ quasiconvex.
This also confirms that the counterexamples in \cite{Zi,Fe-Sou} are only for the case that 
$V$ has small oscillation, in which case $V$ only sees the local structure of $H$ around its strict saddle points, 
but not its global structure.

Let us now state the most general result in this stochastic homogenization context that we have.

\begin{thm}\label{thm:random-m}  Assume that {\rm (H6)} holds for some $m\in \N$. 
Assume further  that  $ {\rm ess}\inf_{\Omega}V(0,\omega)=0$, $ {\rm ess}\sup_{\Omega}V(0,\omega)=\ol{m}$.
Then $\varphi(|p|) - V(x,\om)$ and $k_{m-1}(|p|)- V(x,\om)$ are regularly homogenizable. Moreover,  (\ref{mainfor-1}) and (\ref{mainfor-2}) hold in this random setting as well
\begin{equation*}
 \ol{H}_m=\max\left\{\ol{K}_{m-1}, \  \ol{\Phi}_{2m},\ \varphi(s_{2m})\right\},
\end{equation*}
and
\begin{equation*}
\ol{K}_{m-1}=\min\left\{\ol{H}_{m-1}, \  \ol{\Phi}_{2m-1},\ \varphi(s_{2m-1})-\ol{m}\right\}.
\end{equation*}
In particular,  $ \ol{H}_m$ and $ \ol{K}_{m-1}$ are both even.  Here we use same notations as in Theorem \ref{Maintheorem}. 
\end{thm}

We also have the following conjecture which was proven to be true in one dimension \cite{ATY2}. 

\begin{conj}  \label{conj:ran}
Assume that $\varphi:[0, \infty)\to \R$ is continuous and coercive.  
 Set $H(p,x,\om)=\varphi(|p|)-V(x,\om)$ for $(p,x,\om) \in \R^n \times \R^n \times \Om$. Then $H$ is regularly homogenizable. 
\end{conj}

We believe that  Conjecture \ref{conj:convex}  should play a significant role in proving the above conjecture as in the one dimensional case.  An initial step might be to obtain stochastic homogenization for  the specific $\varphi$ satisfying (H8). Below is closely related elementary question

\begin{quest}\label{quest3}
Let $w$ be a periodic semi-concave (or semi-convex) function.  Denote $\mathcal {D}$ as the collection of all regular gradients, that is, 
$$
\mathcal {D}=\{Dw(x)\,:\,  \text{$w$ is differentiable at $x$}\}.
$$
Is $\mathcal {D}$ a connected set?
\end{quest}
The periodic assumption is essential.  Otherwise, it is obviously  false, e.g.,  $w(x)=-(|x_1|+\cdots +|x_n|)$ for $x=(x_1,\ldots,x_n)\in \R^n$.   When $n=1$,  the connectedness of $\mathcal {D}$  follows easily from the periodicity and a simple mean value property (Lemma 2.6 in \cite{ATY2}).  

As for the double-well type Hamiltonian $H(p)=\min\{|p-e_1|,  \ |p+e_1|\}$ in the two dimensional space,  
the following question is closely related to Question \ref{quest2} and counterexamples in  \cite{Zi,Fe-Sou}.

\begin{quest}\label{quest4}
Assume that $n=2$ and $H(p)=\min\{|p-e_1|,  \ |p+e_1|\}$ for all $p\in \R^2$, where $e_1=(1,0)$.
Does there  exist $L>0$ such that, if 
$$
{\rm osc}_{\R^2\times \Om}V={\rm ess}\sup_{\Omega}V(0,\omega)- {\rm ess}\inf_{\Omega}V(0,\omega)>L,
$$
then $H-V$ is regularly homogenizable?
\end{quest}

\subsection{Proof of Theorem \ref{thm:random}}
As a demonstration,  we only provide the proof of Theorem \ref{thm:random} in details here. 
The extension to Theorem \ref{thm:random-m} is clear.  Compared with the proof for the special case  in \cite{ATY1},  the following  proof is much clearer and simpler.

We need the following comparison result.

\begin{lem}\label{lem:comparison}
Fix $\lam \in (0,1), R>0$.
Suppose that $u,v$ are respectively a viscosity subsolution and a viscosity supersolution to
\begin{equation}\label{eq-com}
\lam w + H(p+Dw) - V= 0 \quad \text{ in } B(0,R/\lam).
\end{equation}
Assume further that there exists $C>0$ such that
\[
\begin{cases} 
\lam(|u|+|v|) \leq C \quad &\text{ on } \ol{B}(0,R/\lam),\\
|H(p)-H(q)| \leq C|p-q| \quad &\text{ for all } p,q \in \R^n.
\end{cases}
\]
Then
\[
\lam (u-v) \leq \frac{C (|x|^2+1)^{1/2}}{R} + \frac{C^2}{R} \quad \text{ in } B(0,R/\lam).
\]
\end{lem}

\begin{proof}
Let 
\[
\tilde v(x) = v(x) +  \frac{C (|x|^2+1)^{1/2}}{R} + \frac{C^2}{R \lam} \quad \text{ for $x\in B(0,R/\lam)$.}
\] 
Then, $\tilde v$ is still a viscosity supersolution to \eqref{eq-com} and furthermore, $\tilde v \geq u$ on $\partial B(0,R/\lam)$.
Hence, the comparison principle yields $\tilde v \geq u$ in $B(0,R/\lam)$.
\end{proof}

\begin{proof}[Proof of Theorem \ref{thm:random}]
Fix $p\in \R^n$.
For $\lam>0$, let $v_\lam(y,p)$ be the unique bounded continuous viscosity solution to
\begin{equation}\label{v-lam}
\lam v_\lam + H(p+Dv_\lam) - V(y) =0 \quad \text{ in } \R^n.
\end{equation}
In order to prove Theorem \ref{thm:random}, it is enough to show that
\begin{equation}\label{random-goal}
\bP \left[ \lim_{\lam\to 0} \left|\lam v_\lam(0,p) + \ol{H}(p)\right|=0\right]=1.
\end{equation}
Let us note first that, as $H_1$ is quasiconvex and $H_2$ is quasiconcave,
$H_1-V$ and $H_2-V$ are regularly homogenizable (see \cite{DaSi, AS3}).
It is clear that
\begin{equation}\label{random-0}
\max\left\{\ol{H}_1, \ol{H}_2 \right\} \leq H.
\end{equation}
Once again, we divide our proof into few steps.
\smallskip

\noindent {\bf Step 1.} Assume that $\ol{H}_1(p) \geq \max\left\{\ol{H}_2(p),0\right\}$. 
We proceed to show that
\begin{equation}\label{random-1-1}
\bP \left[ \lim_{\lam\to 0} \left|\lam v_\lam(0,p) + \ol{H}_1(p)\right|=0\right]=1.
\end{equation}
Since $H \geq H_1$, by the usual comparison principle, it is clear that
\begin{equation}\label{random-1-2}
\bP \left[ \liminf_{\lam\to 0} -\lam v_\lam(0,p) \geq \ol{H}_1(p)\right]=1.
\end{equation}
It suffices to show that
\begin{equation}\label{random-1-3}
\bP \left[ \limsup_{\lam\to 0} -\lam v_\lam(0,p) \leq \ol{H}_1(p)\right]=1.
\end{equation}
Let $v_{1\lam}(y,-p)$ be the viscosity solution to
\begin{equation}\label{random-1-4}
\lam v_{1\lam} + H_1(-p+Dv_{1\lam}) - V =0 \quad \text{ in } \R^n.
\end{equation}
Since $H_1-V$ is regularly homogenizable, we get that, for any $R>0$,
\begin{equation}\label{random-1-5}
\bP \left[ \lim_{\lam\to 0} \max_{y \in B(0,R/\lam)} \left|\lam v_{1\lam}(y,-p) + \ol{H}_1(-p)\right|=0\right]=1.
\end{equation}
Fix $R>0$. Pick $\om \in \Om$ such that \eqref{random-1-5} holds.
For each $\ep>0$ sufficiently small, there exists $\lam(R,\om,\ep)>0$ such that, for $\lam <\lam(R,\om,\ep)$,
\[
 \max_{y \in B(0,R/\lam)} \left|\lam v_{1\lam}(y,-p,\om) + \ol{H}_1(-p)\right| \leq \ep.
\]
Note that by inf-sup representation formula and the even property of $H_1$, we also have that $\ol{H}_1$ is also even, i.e., $\ol{H}_1(-p) = \ol{H}_1(p)$.
In particular,
\[
-\lam v_{1\lam}(y,-p,\om) \geq \ol{H}_1(p) - \ep \geq -\ep \quad \text{ for } y \in B(0,R/\lam).
\]
Due to the quasiconvexity of $H_1$,  this implies that, for $q \in D^- v_{1\lam}(y,-p,\om)$ for some $y \in B(0,R/\lam)$, 
\begin{equation}\label{random-1-6}
H_1(-p+q) = -\lam v_{1\lam}(y,-p,\om) +V(y) \geq -\ep.
\end{equation}
In particular, $|H_1(-p+q)-H(-p+q)|\leq \del_\ep$, where $\del_\ep = \max\{H(p)- H_1(p)\,:\, H_1(p) \geq -\ep\}$.

Denote by $w=-v_{1\lam}(y,-p,\om)  - \frac{2 \ol{H}_1(p)}{\lam}$. Then $w$ is a viscosity subsolution to
\[
\lam w + H(p+Dw) - V = 2\ep + \del_\ep \quad \text{ in } B(0,R/\lam).
\]
Hence $w - \frac{2\ep+\del_\ep}{\lam}$ is a subsolution to \eqref{eq-com}.
By Lemma \ref{lem:comparison}, we get
\[
\lam w(0) - \lam v_\lam(0,p,\om) \leq \frac{C}{R} +2 \ep + \del_\ep.
\]
Hence, \eqref{random-1-3} holds.
Compare this to Step 2 in the proof of Theorem \ref{thm:rep1} for similarity.
\smallskip

\noindent {\bf Step 2.} Assume that $\ol{H}_2(p) \geq \max\left\{\ol{H}_1(p),0\right\}$.
We proceed in the same way as in Step 1 (except that we use $v_{2\lam}(y,p)$ instead of $v_{2\lam}(y,-p)$
because of the quasiconcavity of $H_2$) to get that
\begin{equation}\label{random-2-1}
\bP \left[ \lim_{\lam\to 0} \left|\lam v_\lam(0,p) + \ol{H}_2(p)\right|=0\right]=1.
\end{equation}
Compare this to Step 3 in the proof of Theorem \ref{thm:rep1} for similarity.
\smallskip

\noindent {\bf Step 3.} We now consider the case $\max\left\{\ol{H}_1(p), \ol{H}_2(p) \right\} <0$.
Our goal is to show 
\begin{equation}\label{random-3-1}
\bP \left[ \lim_{\lam\to 0} \left|\lam v_\lam(0,p)\right|=0\right]=1.
\end{equation}
This step basically shares the same philosophy as Step 4 in the proof of Theorem \ref{thm:rep1}.
Let us still present a proof here.

Thanks to the assumption that ${\rm ess}\inf_{\Om} V(0,\om)=0$ and the fact that $H \geq 0$, 
\begin{equation}\label{random-3-2}
\bP \left[ \liminf_{\lam\to 0} -\lam v_\lam(0,p) \geq 0\right]=1.
\end{equation}
We therefore only need to show
\begin{equation}\label{random-3-2}
\bP \left[ \limsup_{\lam\to 0} -\lam v_\lam(0,p) \leq 0\right]=1.
\end{equation}

For $\sig \in [0,1]$, $H_i - \sig V$ are still regularly homogenizable for $i=1,2$.
Let $\ol{H}_i^\sig$ be the effective Hamiltonian corresponding to $H_i-\sig V$.
By repeating Steps 1 and 2 above, we get that:
\begin{equation}\label{random-3-3}
\begin{cases}
\text{For $p \in \R^n$ and $\sig \in [0,1]$, if $\max\left\{ \ol{H}_1^\sig(p),\ol{H}_2^\sig(p) \right\}=0$,}\\
\text{then $H-\sig V$ is regularly homogenizable at $p$ and $\ol{H}^\sig(p)=0$},
\end{cases}
\end{equation}
where $\ol{H}^\sig(p)$ is its corresponding effective Hamiltonian.
Take $p\in \R^n$ so that 
\begin{equation}\label{random-3-4}
\max\left\{ \ol{H}_1^\sig(p),\ol{H}_2^\sig(p) \right\}=0.
\end{equation}
for some $\sigma\in [0,1]$.   For $\lam>0$, let $v_\lam^\sig$ be the viscosity solution to
\[
\lam v_\lam^\sig + H(p+Dv_\lam^\sig) - \sig V =0 \quad \text{ in } \R^n,
\]
then by  \eqref{random-3-3}, 
\[
\bP \left[ \lim_{\lam\to 0} \left|\lam v_\lam^\sig(0,p)\right|=0\right]=1.
\]
Since $V\geq 0$, the usual comparison principle gives us that $v_\lam^\sig \leq v_\lam$. Hence, \eqref{random-3-2} holds true.

It remains to show that, if $\max\left\{\ol{H}_1(p), \ol{H}_2(p) \right\} <0$, then \eqref{random-3-4} holds for some $\sig \in [0,1]$.
As $\ol{H}^0(p)= H(p)=0$ for $p \in \partial U$, we only need to consider the case $p \notin \partial U$.
There are two cases, either $p \in U$ or $p\in \R^n \setminus \ol{U}$.
Again, it is enough to consider the case that $p \in U$.
For this $p$, we have that
\[
\ol{H}_2^0(p) = H_2(p)>0 \quad \text{and} \quad \ol{H}_2^1(p)=\ol{H}_2(p) <0.
\]
By the continuity of $\sig \mapsto \ol{H}_2^\sig(p)$, there exists $\sig \in [0,1]$ such that $\ol{H}_2^\sig(p)=0$.
This, together with the fact that $\ol{H}_1^\sig(p)\leq  H_1(p)<0$,  leads to  \eqref{random-3-4}.
\end{proof}

\bibliographystyle{plain}

\end{document}